\documentclass[a4paper,10pt]{article}

\usepackage{amsfonts}
\usepackage{amssymb}
\usepackage{amsmath}
\usepackage{hyperref}
\usepackage{cleveref}
\usepackage{amsthm}
\usepackage{graphicx}
\usepackage[caption=false]{subfig}
\usepackage{tikz}

\title{{Voronoi Cells of Lattices with Respect to Arbitrary Norms}
}

\author{
  Johannes Bl\"omer
  \and
  Kathl\'en Kohn
}

\usepackage{amsopn}
\DeclareMathOperator{\vol}{vol}

\DeclareMathOperator{\intt}{int}

\newcommand{\svp}{\textrm{\sc Svp}}
\newcommand{\cvp}{\textrm{\sc Cvp}}

\theoremstyle{plain}
\newtheorem{theorem}{Theorem}
\newtheorem{proposition}[theorem]{Proposition}
\newtheorem{corollary}[theorem]{Corollary}
\newtheorem{lemma}[theorem]{Lemma}

\theoremstyle{definition}
\newtheorem{definition}[theorem]{Definition}

\theoremstyle{remark}
\newtheorem{claim}{Claim}

\newtheorem{proofofmaintheorem}{Proof of Theorem~\ref{thm:main}}
\newtheorem{proofoftrivialbound}{Proof of Proposition~\ref{prop:trivial_upper_bound}}
\newtheorem{proofofVRinduceFacets}{Proof of Proposition~\ref{prop:Voronoi_relevant_has_facet}}
\newtheorem{proofof6VR}{Proof of Proposition~\ref{prop:2dimensions_strictly_convex_exact}}
\newtheorem{proofofVRdefineVC}{Proof of Theorme~\ref{thm:Voronoi-relevant_vectors_define_Voronoi-cell}}
\newtheorem{proofofbijection}{Proof of Theorem~\ref{thm:bijection_facets_Voronoi-relevant}}
\newtheorem{proofofgenVRcomb}{Proof of Proposition~\ref{prop:2dimensions_1norm}}
\newtheorem{proofofGVRdefineVC}{Proof of Theorem~\ref{thm:generalized_Voronoi-relevant_vectors_define_Voronoi-cell}}
\newtheorem{proofofclaim}{Proof}

\begin{document}

\maketitle

\begin{abstract}
We study the geometry and complexity
of Voronoi cells of lattices with respect to arbitrary norms. On the
positive side, we show for strictly convex and smooth norms that the
geometry of Voronoi cells of lattices in any dimension is similar to the
Euclidean case, i.e., the Voronoi cells are defined by the so-called
Voronoi-relevant vectors and the facets of a Voronoi cell
are in one-to-one correspondence with these
vectors. On the negative side, we show that Voronoi
cells are combinatorially much more
complicated for arbitrary strictly convex and smooth norms than in the Euclidean case. In particular, we construct a
family of three-dimensional lattices whose number of
Voronoi-relevant vectors with respect to the $\ell_3$-norm is
unbounded. Our result indicates, that the break through single exponential time algorithm of Micciancio and Voulgaris for solving the shortest and closest vector problem in the Euclidean norm cannot be extended to achieve deterministic
single exponential time algorithms for lattice problems with respect
to arbitrary $\ell_p$-norms. In fact, the algorithm of Micciancio and Voulgaris and its
run time analysis crucially depend on the fact that for the Euclidean
norm the number of Voronoi-relevant vectors is single exponential in
the lattice dimension.
\end{abstract}

\section{Introduction}
We study the geometry and complexity of Voronoi cells of lattices with respect to arbitrary norms. Our original motivation for studying these problems stemmed from recent algorithms to solve problems in the geometry of numbers. However, the questions we consider in this paper are interesting and fundamental research questions in the geometry of numbers, that so far have received surprisingly little attention.  
%

A \emph{lattice} is a discrete additive subgroup $\Lambda$ of $\mathbb{R}^n$. Its \emph{rank} is the dimension of the $\mathbb{R}$-subspace it spans. For simplicity, we only consider lattices of \emph{full rank}, i.e., with rank $n$. Each such lattice $\Lambda$ has a basis $(b_1, \ldots, b_n)$ such that $\Lambda$ equals
\begin{align*}
\mathcal{L}(b_1, \ldots, b_n) :=
\left\lbrace \sum \limits_{i=1}^n z_i b_i \;\middle\vert\; z_1, \ldots, z_n \in \mathbb{Z} \right\rbrace.
\end{align*}
The most important computational lattice problems are the Shortest Vector Problem ($\svp$) and the Closest Vector Problem ($\cvp$). Both problems have numerous applications in combinatorial optimization, complexity theory, communication theory, and cryptography. In $\svp$, given a lattice $\Lambda={\mathcal L}(b_1,\ldots,b_n)$, a non-zero vector in $\Lambda$ with minimum length has to be computed. In $\cvp$, given a lattice $\Lambda$ and a target vector $t\in \mathbb{R}^n$, a vector in $\Lambda$  with  minimum distance to $t$ has to be computed.  Both, $\svp$ and $\cvp$ are best studied for the Euclidean norm, but for applications other norms like polyhedral norms or general $\ell_p$-norms are important as well. 

In particular, when studying $\cvp$, an important and natural geometric object defined by a lattice and a norm, is the \emph{Voronoi cell} or \emph{Voronoi region}. For a lattice $\Lambda$ and a norm $\Vert\cdot\Vert$, this is defined as 
\begin{align}\label{def-VC}
\mathcal{V}(\Lambda, \Vert \cdot \Vert) := \lbrace x \in \mathbb{R}^n \mid \forall v \in \Lambda: \Vert x \Vert \leq \Vert x-v \Vert \rbrace,
\end{align}
i.e., it is the the set of points closer to the origin than to any other lattice point (with respect to the given norm). For the Euclidean, or $\ell_2$, norm Voronoi cells of lattices have been studied intensively. For arbitrary norms, on the other hand, very little is known about the geometrical and combinatorial properties of Voronoi cells. 

\paragraph{Properties of Voronoi cells} It is easy to see that for a lattice $\Lambda$ and any norm $\Vert\cdot \Vert$ the Voronoi cell $\mathcal{V}(\Lambda, \Vert \cdot \Vert)$ is centrally symmetric and star-shaped. For the Euclidean norm, the Voronoi cell of a lattice in $\mathbb{R}^n$ is a polytope  with at most $2(2^{n}-1)$ facets. In particular, in the definition of the Voronoi cell of a lattice in $\mathbb{R}^n$ only $2(2^{n}-1)$ lattice vectors have to considered in (\ref{def-VC}), namely those vectors $v$ for which there exists a point that is closer to $v$ and the origin $0$ than to any other lattice vector. Such vectors are called \emph{Voronoi-relevant}. For the Euclidean norm, every facet of the Voronoi cell of a lattice is contained in a bisector between $0$ and a Voronoi-relevant vector. 
These are classical results going back to Minkowski and Voronoi (see for example~\cite{conway2013sphere} and~\cite{AEVZ02}). Recently it has also been shown that it is computationally hard to determine the exact number of facets of  the Voronoi cell of a lattice, i.e., the problem is $\#$P hard (see~\cite{dutour2009complexity}).

For Voronoi cells of lattices with respect to norms other than the Euclidean norm much less is known. For an arbitrary norm Voronoi cells of lattices, in general, are not convex. More precisely, a result by Mann~\cite{mann1935untersuchungen} states that if for all lattices the Voronoi cells with the respect to a given norm are convex, then the norm is Euclidean. For strictly convex norms (to be defined below) Horv\'{a}rth showed that bisectors are homeomorphic to hyperplanes~\cite{bisector_hyperplane}. But, prior to this work, it was not known whether Voronoi-relevant vectors and bisectors between Voronoi-relevant vectors and $0$ determine the Voronoi cell of a lattice with respect to a strictly convex norm. For norms that are not strictly convex, like polyhedral norms, bisectors are not homeomorphic to hyperplanes. This implies that for these norms the geometry of Voronoi cells differs significantly from the geometry of Voronoi cells with respect to strictly convex norms like the Euclidean norm. In~\cite{deza2015voronoi} Deza et al. describe algorithms to compute the Voronoi cell of a lattice with respect to a polyhedral norm.   Finally, the only upper bound on the combinatorial complexity of Voronoi cells of lattices with respect to arbitrary norms that we are aware of bounds the number of Voronoi-relevant vectors of a lattice in terms of the rank, the covering radius, and the first successive minimum of the lattice (see Theorem~\ref{prop:trivial_upper_bound} for
 details). Finally, let us remark that for finite point sets, instead of lattices, much more is known about their Voronoi cells, or Voronoi diagrams, with respect to arbitrary norms. See for example the surveys~\cite{martini2004geometry,aurenhammer1991voronoi}.

\paragraph{Voronoi cells and the Micciancio Voulgaris algorithm}
 The question whether the algorithm by Micciancio and Voulgaris for solving $\svp$ and $\cvp$ for the Euclidean norm~\cite{MV13} can be generalized to more general norms was our original motivation for studying Voronoi cells of lattices with respect to arbitrary norms.

$\svp$ and $\cvp$ are computationally hard problems. If $\textbf{P}\ne \textbf{NP}$, then for all $\ell_p$-norms ($p \in \mathbb{R}_{>1}$) and lattices with rank $n$, $\cvp$ cannot be solved approximately with factor $n^{c/\log\log(n)}$ for some $c>0$~\cite{DKS98}. Under stronger, but still reasonable assumptions, the same inapproximability result holds for $\svp$~\cite{haviv2012tensor}. Currently, algorithms by Micciancio and Voulgaris~\cite{MV13} and by Aggarwal et al.~\cite{ADS15} are the only deterministic single exponential time (in the rank of the lattice) algorithms solving $\cvp$ exactly. Both algorithms only work for the Euclidean norm. The single exponential time and space complexity of the algorithm by Micciancio and Voulgaris crucially depends on the above mentioned fact that for any lattice in $\mathbb{R}^n$ the number of Voronoi-relevant vectors is bounded by $2(2^n-1)$. The algorithm presented in~\cite{ADS15} uses Gaussian-like distributions on lattices. There exist other single exponential time algorithms to solve $\cvp$ with approximation factor $(1+\epsilon),\epsilon>0$ arbitrary. Some of these algorithms work for general norms or even for semi-norms~\cite{AKS02,BN09,DPV11, DV12, DK13}. Hence the main open question in this area is whether there exist single exponential time and space algorithms solving $\cvp$ exactly for norms other than the Euclidean norm. In particular, Micciancio and Voulgaris already asked whether their techniques can be generalized and extended to obtain such an algorithm. One of the main results in this paper shows that, without significant modifications and extensions, this is unlikely to be possible. 
On the positive side, we show that the structure of Voronoi cells of lattices with respect to many norms resembles the situation in the Euclidean case. However, we also show that already in $\mathbb{R}^3$ and for the $\ell_3$-norm, the complexity of the Voronoi cells differs dramatically from the Euclidean case. More precisely, we construct a sequence of lattices $\Lambda_k, k\in \mathbb{Z}_{>0}$, such that $\Lambda_k$ has at least $k$ Voronoi-relevant vectors. Hence, without further restrictions and significant modifications it seems unlikely that the techniques by Micciancio and Voulgaris can be generalized to obtain deterministic single exponential time and space algorithms for $\svp$ and $\cvp$ with respect to norms other than the Eucidean norm. Note, that although our results indicate that the algorithm by Micciancio and Voulgaris cannot be generalized directly to norms other than the Euclidean norm, it is still useful to obtain single exponential time algorithms for versions of lattice problems with respect to norms other than the Euclidean norm. In fact, the single exponential time algorithms for  approximate versions of $\cvp$ in general norms mentioned above (\cite{DV12,DK13}), rely among other techniques, on Micciancio's and Voulgaris' algorithm. 

Although we present our construction for lattices with an unbounded number of Voronoi-relevant vectors only for the $\ell_3$-norm, it will be clear from the construction that similar results hold for more general norms, e.g., all $\ell_p$-norms for $p \in \mathbb{R}_{>2}$.

\section{Overview of Results and Roadmap}
\label{sec:results}

We show that not all lattice vectors have to be considered in the definition of Voronoi cells in (\ref{def-VC}), but that a finite set of vectors is enough, namely for strictly convex and smooth norms $\Vert\cdot\Vert$ the  \emph{Voronoi-relevant vectors} are sufficient. For general norms, \emph{weak Voronoi-relevant vectors} have to be considered.
\begin{definition}
\label{defn:Voronoi_relevant}
Let $\Vert \cdot \Vert: \mathbb{R}^n \to \mathbb{R}_{\geq 0}$ be a norm and let $\Lambda \subseteq \mathbb{R}^n$ be a lattice.
\begin{itemize}
  \item  A lattice vector $v \in \Lambda \setminus \lbrace 0 \rbrace$ is a \emph{Voronoi-relevant vector} if there is some $x \in \mathbb{R}^n$ such that $\Vert x \Vert = \Vert x-v \Vert < \Vert x-w \Vert$ holds for all $w \in \Lambda \setminus \lbrace 0,v \rbrace$.
  \item A lattice vector $v \in \Lambda \setminus \lbrace 0 \rbrace$ is a \emph{weak Voronoi-relevant vector} if there is some $x \in \mathbb{R}^n$ such that $\Vert x \Vert = \Vert x-v \Vert \leq \Vert x-w \Vert$ holds for all $w \in \Lambda$.
\end{itemize}
\end{definition}
A norm $\Vert \cdot \Vert: \mathbb{R}^n \to \mathbb{R}_{\geq 0}$ is said to be \emph{strictly convex} if for all distinct $x,y \in \mathbb{R}^n$ with $\Vert x \Vert = \Vert y \Vert$ and all $\tau \in (0,1)$ we have that $\Vert \tau x + (1-\tau)y \Vert < \Vert x \Vert$.
We give a formal definition for smooth norms in Section~\ref{sec:geom}.
Now, we formulate our first result on Voronoi cells precisely.
\begin{theorem}
\label{thm:Voronoi-relevant_vectors_define_Voronoi-cell}
For every lattice $\Lambda \subseteq \mathbb{R}^n$ and every strictly convex and smooth norm~$\Vert \cdot \Vert$, the Voronoi cell $\mathcal{V}(\Lambda, \Vert \cdot \Vert)$ is equal to
\begin{align*}
\tilde{\mathcal{V}}(\Lambda, \Vert \cdot \Vert)
:= \left\lbrace x \in \mathbb{R}^n \;\middle\vert\; \begin{array}{l}
\forall v \in \Lambda \text{ Voronoi-relevant with} \\
\text{respect to } \Vert \cdot \Vert:  \Vert x \Vert \leq \Vert x-v \Vert
\end{array} \right\rbrace.
\end{align*}
For two-dimensional lattices, smoothness of the norm is not necessary.
\end{theorem}
We will see that Theorem~\ref{thm:Voronoi-relevant_vectors_define_Voronoi-cell} does not hold for non-strictly convex norms, not even in two dimensions.
Instead we prove the following weaker result.
\begin{theorem}
\label{thm:generalized_Voronoi-relevant_vectors_define_Voronoi-cell}
For every lattice $\Lambda \subseteq \mathbb{R}^n$ and every norm $\Vert \cdot \Vert$, we have that the Voronoi cell $\mathcal{V}(\Lambda, \Vert \cdot \Vert)$ is equal to
\begin{align*}
\tilde{\mathcal{V}}^{(g)}(\Lambda, \Vert \cdot \Vert)
:= \left\lbrace x \in \mathbb{R}^n \;\middle\vert\; \begin{array}{l}
\forall v \in \Lambda \text{ weak Voronoi-relevant} \\
\text{with respect to } \Vert \cdot \Vert:  \Vert x \Vert \leq \Vert x-v \Vert
\end{array} \right\rbrace.
\end{align*}
\end{theorem}
When considering the complexity of the Voronoi cell of a given lattice, we are particularly interested in the number of \emph{facets} of that Voronoi cell.
Intuitively, such a facet is an at least $(n-1)$-dimensional boundary part of the Voronoi cell containing all points that have the same distance to 0 than to some fixed non-zero lattice vector. 
For the formal definition we denote for a given norm $\Vert \cdot \Vert$ on $\mathbb{R}^n$, a point $x \in \mathbb{R}^n$ and $\delta \in \mathbb{R}_{>0}$, the (open) \emph{$\Vert \cdot \Vert$-ball} around $x$ with radius $\delta$ by $\mathcal{B}_{\Vert \cdot \Vert, \delta}(x) := \lbrace y \in \mathbb{R}^n \mid \Vert y-x \Vert < \delta \rbrace$.
The \emph{bisector} of $a,b \in \mathbb{R}^n, a \neq b$ is $\mathcal{H}_{\Vert \cdot \Vert}^=(a,b) := \lbrace x \in \mathbb{R}^n \mid \Vert x-a \Vert = \Vert x-b \Vert \rbrace$.
\begin{definition}
\label{defn:facet}
For a lattice $\Lambda \subseteq \mathbb{R}^n$ and a norm $\Vert \cdot \Vert: \mathbb{R}^n \to \mathbb{R}_{\geq 0}$, a \emph{facet} of the Voronoi cell is a subset $\mathcal{F} \subseteq \mathcal{V}(\Lambda, \Vert \cdot \Vert)$ such that
\begin{enumerate}
  \item $\exists v \in \Lambda \setminus \lbrace 0 \rbrace : \mathcal{F} = \mathcal{V}(\Lambda, \Vert \cdot \Vert) \cap \mathcal{H}_{\Vert \cdot \Vert}^{=} (0,v)$,
\item $\exists x \in \mathcal{F} \, \exists \delta \in \mathbb{R}_{>0}:
\mathcal{B}_{\Vert \cdot \Vert_2, \delta}(x) \cap \mathcal{H}_{\Vert \cdot \Vert}^{=}(0,v) \subseteq \mathcal{F}$,
\item there is no $w \in \Lambda \setminus \lbrace 0 \rbrace$ such that $\mathcal{F} \subsetneqq  \mathcal{V}(\Lambda, \Vert \cdot \Vert) \cap \mathcal{H}_{\Vert \cdot \Vert}^{=} (0,w)$.
\end{enumerate}
\end{definition}
We will see that for strictly convex and smooth norms (or only strictly convex norms in the two-dimensional case) the second condition implies the third condition. For non-strictly convex norms it can happen that $(n-1)$-dimensional bisector parts of a Voronoi cell are contained in each other (see Figure~\ref{fig:facets}).
For every norm we show as our first result on facets that the complexity of Voronoi cells is lower bounded by the number of Voronoi-relevant vectors, i.e., every Voronoi-relevant vector defines a facet of the Voronoi cell.
\begin{proposition}
\label{prop:Voronoi_relevant_has_facet}
Let $\Vert \cdot \Vert: \mathbb{R}^n \to \mathbb{R}_{\geq 0}$ be a norm and let $\Lambda \subseteq \mathbb{R}^n$ be a lattice. For every lattice vector $v \in \Lambda$ which is Voronoi-relevant with respect to $\Vert \cdot \Vert$ we have that $\mathcal{V}(\Lambda, \Vert \cdot \Vert) \cap \mathcal{H}_{\Vert \cdot \Vert}^{=}(0,v)$ is a facet of the Voronoi cell.
\end{proposition}
The next theorem shows that for every smooth and strictly convex norm there is a $1$-to-$1$ correspondence between Voronoi-relevant vectors and the facets of a Voronoi cell.
\begin{theorem}
\label{thm:bijection_facets_Voronoi-relevant}
For every lattice $\Lambda$ and every strictly convex and smooth norm $\Vert \cdot \Vert$, we have that $v \mapsto \mathcal{V}(\Lambda, \Vert \cdot \Vert) \cap \mathcal{H}_{\Vert \cdot \Vert}^{=}(0,v)$ is a bijection between Voronoi-relevant vectors and facets of the Voronoi cell.  For two-dimensional lattices, smoothness of the norm is not necessary.
\end{theorem}
For the Euclidean norm, every lattice in $\mathbb{R}^n$ has at most $2(2^n-1)$ Voronoi-relevant vectors. 
By the above bijection, this is also a bound for the complexity of Voronoi cells under the Euclidean norm.
Unfortunately, for other norms and lattices in dimension three and higher, we cannot expect the complexity of Voronoi cells to depend only on the dimension and possibly the norm. In fact, already for $n=3$ and the $\ell_3$-norm we obtain a family of lattices with arbitrarily many Voronoi-relevant vectors. In particular, this implies that Micciancio's and Voulgaris' algorithm cannot be directly generalized to $\ell_p$-norms (for $p \in (1,\infty)$) without exceeding its single exponential running time.
\begin{theorem}
\label{thm:general_dimensions}
For every $k \in \mathbb{Z}_{>0}$, there is a lattice $\Lambda_k\subseteq \mathbb{R}^3$ with at least $k$ Voronoi-relevant vectors with respect to the 
(smooth and strictly convex) $\ell_3$-norm $\Vert \cdot \Vert_3$.
\end{theorem}
\begin{corollary}
For lattices $\Lambda \subseteq \mathbb{R}^n, n\ge 3$ and a strictly convex and smooth norm $\Vert\cdot\Vert$, the number of Voronoi-relevant vectors of $\Lambda$ with respect to $\Vert\cdot\Vert$ cannot be bounded by a function only depending on $n$ and $\Vert\cdot\Vert$.
\end{corollary}
In $\mathbb{R}^2$, Voronoi cells with respect to strictly convex norms behave as Voronoi cells with respect to the Euclidean norm $\Vert\cdot\Vert_2$. This we have seen geometrically in Theorems~\ref{thm:Voronoi-relevant_vectors_define_Voronoi-cell} and~\ref{thm:bijection_facets_Voronoi-relevant}. In the next proposition we show that for strictly convex norms Voronoi cells  of lattices in $\mathbb{R}^2$ have at most six facets. Note that this coincides with the upper bound $2(2^n-1)$ for the complexity of Voronoi-cells of lattices in $\mathbb{R}^2$ with respect to the  Euclidean norm.
\begin{proposition}
\label{prop:2dimensions_strictly_convex_exact} 
For every lattice $\Lambda \subseteq \mathbb{R}^2$ and every strictly convex norm  $\Vert \cdot \Vert$, we have that $\Lambda$ has either $4$ or $6$ Voronoi-relevant vectors with respect to $\Vert \cdot \Vert$.
\end{proposition}
  Under non-strictly convex norms on $\mathbb{R}^2$, the Voronoi-relevant vectors generally do not define the Voronoi cell. Instead we can use weak Voronoi-relevant vectors to describe Voronoi cells (Theorem~\ref{thm:generalized_Voronoi-relevant_vectors_define_Voronoi-cell}), but the number of these vectors is generally not bounded by a constant.
\begin{proposition}
\label{prop:2dimensions_1norm}
For every $k \in \mathbb{Z}_{>0}$, there is a lattice $\Lambda_k \subseteq \mathbb{R}^2$ with at least $k$ weak Voronoi-relevant vectors with respect to $\Vert \cdot \Vert_1$.
\end{proposition}
At least we can give an upper bound for the number of weak Voronoi-relevant vectors with respect to arbitrary norms using more refined lattice parameters than simply the lattice dimension.
More precisely, in addition to the dimension, the upper bound depends on the ratio of the covering radius
$\mu(\Lambda, \Vert \cdot \Vert) :=
\inf \lbrace d \in \mathbb{R}_{\geq 0} \mid \forall x \in \mathbb{R}^n \, \exists v \in \Lambda: \Vert x-v \Vert \leq d \rbrace$
 and the length $\lambda_1(\Lambda, \Vert \cdot \Vert)$ of a shortest non-zero lattice vector. 
The value $\lambda_1(\Lambda, \Vert \cdot \Vert)$ is also known as the first successive minimum of $\Lambda$.
\begin{proposition}
\label{prop:trivial_upper_bound}
For every lattice $\Lambda \subseteq \mathbb{R}^n$ and every norm $\Vert \cdot \Vert$, the lattice~$\Lambda$ has at most $\left( 1+4 \frac{\mu(\Lambda, \Vert \cdot \Vert)}{\lambda_1(\Lambda, \Vert \cdot \Vert)} \right)^n$ weak Voronoi-relevant vectors with respect to $\Vert \cdot \Vert$.
\end{proposition}
Although this result seems to be folklore, below we provide a proof for it. An important open question is if one actually can construct a family of lattices $\Lambda_n \subseteq \mathbb{R}^n$ whose number of (weak) Voronoi-relevant vectors grows as $\Theta \left( \left(\frac{\mu(\Lambda, \Vert \cdot \Vert)}{\lambda_1(\Lambda, \Vert \cdot \Vert)} \right)^n \right)$.

Another open problem is if the smoothness assumption in Theorems~\ref{thm:Voronoi-relevant_vectors_define_Voronoi-cell} and~\ref{thm:bijection_facets_Voronoi-relevant} can be omitted.
We need this assumption in higher dimensions due to our proof techniques, which use manifolds and norms that are continuously differentiable as functions.
With this and the Regular Level Set Theorem (see Proposition~\ref{prop:regular_level_set_theorem}), we prove in Proposition~\ref{prop:trisector_is_(n-2)-dimensional} that the intersection of bisectors of three non-collinear points in $\mathbb{R}^n$ is an $(n-2)$-dimensional manifold.
Non-smooth norms are not continuously differentiable on the whole $\mathbb{R}^n \setminus \lbrace 0 \rbrace$, and thus we cannot apply the Regular Level Set Theorem to derive Proposition~\ref{prop:trisector_is_(n-2)-dimensional}.

\paragraph{Organization}
Section~\ref{sec:comb} is devoted to the proof of our main result Theorem~\ref{thm:general_dimensions}.
Moreover, we prove Proposition~\ref{prop:trivial_upper_bound}.
The second part of this paper in Section~\ref{sec:geom} considers relations between the (weak) Voronoi-relevant vectors of a lattice and its Voronoi cell.
In Subsection~\ref{ssec:genNorms}, we discuss results which hold for all norms, as Theorem~\ref{thm:generalized_Voronoi-relevant_vectors_define_Voronoi-cell} and Proposition~\ref{prop:Voronoi_relevant_has_facet}.
We study strictly convex norms in Subsection~\ref{ssec:strNorms} and show Theorems~\ref{thm:Voronoi-relevant_vectors_define_Voronoi-cell} and~\ref{thm:bijection_facets_Voronoi-relevant}.
Finally, in Subsection~\ref{ssec:2Dim}, we focus on two-dimensional lattices and show Propositions~\ref{prop:2dimensions_strictly_convex_exact} and~\ref{prop:2dimensions_1norm}.

\section{Lower Bound on the Complexity of Voronoi cells}
\label{sec:comb}
A first approach to the number of (weak) Voronoi-relevant vectors is stated in Proposition~\ref{prop:trivial_upper_bound}, which in particular implies that there are always finitely many weak Voronoi-relevant vectors.
\begin{proofoftrivialbound}
The proof uses an easy packing argument. By the definition of weak Voronoi-relevant vectors it holds for every such vector $v$ that $\Vert v \Vert \leq 2 \mu (\Lambda, \Vert \cdot \Vert)$. Thus, we have 
\begin{align*}
  \bigcup \limits_{\substack{v \in \Lambda \text{ weak} \\ \text{Voronoi-relevant}}} \mathcal{B}_{\Vert \cdot \Vert, \frac{\lambda_1(\Lambda, \Vert \cdot \Vert)}{2}}(v) \subseteq 
  \mathcal{B}_{\Vert \cdot \Vert, 2 \mu (\Lambda, \Vert \cdot \Vert) + \frac{\lambda_1(\Lambda, \Vert \cdot \Vert)}{2}}(0),
\end{align*}
where the left union is disjoint by definition of $\lambda_1$. This shows that the number of weak Voronoi-relevant vectors is upper bounded by
\begin{align*}
  \frac{\vol \left(\mathcal{B}_{\Vert \cdot \Vert, 2 \mu (\Lambda, \Vert \cdot \Vert) + \frac{\lambda_1(\Lambda, \Vert \cdot \Vert)}{2}}(0) \right)}
  {\vol \left(\mathcal{B}_{\Vert \cdot \Vert, \frac{\lambda_1(\Lambda, \Vert \cdot \Vert)}{2}}(0) \right)}
  = \left( 1+4 \frac{\mu(\Lambda, \Vert \cdot \Vert)}{\lambda_1(\Lambda, \Vert \cdot \Vert)} \right)^n.
  \hfill \qed
\end{align*}
\end{proofoftrivialbound}

Now we prove Theorem~\ref{thm:general_dimensions} by constructing a family of three-dimensional lattices such that their number of Voronoi-relevant vectors with respect to the $\ell_3$-norm $\Vert \cdot \Vert_3$ is not bounded from above by a constant.
The idea is to use a lattice of the form $\mathcal{L}(e_1, e_2, M e_3)$, where $(e_1, e_2, e_3)$ denotes the standard basis of $\mathbb{R}^3$ and $M \in \mathbb{Z}_{>0}$ is chosen sufficiently large, and to apply some rotations to this lattice.
These rotations will depend on a parameter $m \in \mathbb{Z}_{>0}$ such that every lattice in the family is rotated differently.
The basis vectors of the rotated lattices will be denoted by $b_{m,1}$, $b_{m,2}$ and $b_{m,3}$ and will coincide with the rotated versions of $e_1$, $e_2$ and $Me_3$, respectively. The intuition is to rotate $\mathcal{L}(e_1, e_2, M e_3)$ such that the line segment between $0$ and $b_{m,1} + m b_{m,2}$ lies in an \emph{edge} of a scaled and translated unit ball of the $\ell_3$-norm when intersecting the plane spanned by $0$, $b_{m,1}$ and $b_{m,2}$ with the ball.

Figure~\ref{fig:3norm} shows the closed unit ball of the $\ell_3$-norm with and without intersections with different planes. The unit ball can be intuitively seen as a cube with rounded edges and corners. Throughout the following description, we will make often use of this intuitive notion of an edge of the unit ball.
Let the $x$-, $y$- and $z$-axis denote the axes of the standard three-dimensional coordinate system which are spanned by $e_1$, $e_2$ and $e_3$, respectively.
As seen in Figures~\ref{sfig:norm_plane_ortho} and~\ref{sfig:norm_plane_ortho_intersect}, the intersection of the ball with a plane which is orthogonal to the $z$-axis (e.g., the plane spanned by $0$, $e_1$ and $e_2$) yields a scaled unit ball of the $\ell_3$-norm in two dimensions.
But when such a plane is rotated around the $y$-axis by $45^\circ$, as in Figures~\ref{sfig:norm_plane_45} to~\ref{sfig:norm_plane_45extreme_intersect}, it intersects the three-dimensional unit ball of the $\ell_3$-norm at one of its edges.  These kinds of intersections are roughly speaking as less circular as possible, and the closer the plane is to the edge, the less circular the intersection is. 
With \emph{less circular} we mean that the ratio between the diameter of the smallest circle in the plane containing the intersection and the diameter of the largest circle contained in the intersection is large.
Due to this intuition, the plane spanned by $0$, $b_{m,1}$ and $b_{m,2}$ should be of the form of the plane in Figures~\ref{sfig:norm_plane_45extreme} and~\ref{sfig:norm_plane_45extreme_intersect}. Moreover, the line segment between $0$ and $b_{m,1} + m b_{m,2}$ should lie directly on the edge of a scaled and translated unit ball such that all other lattice points in the plane spanned by $0$, $b_{m,1}$ and $b_{m,2}$ lie outside of the ball. This is illustrated in Figure~\ref{fig:plane_through_ball} for the case $m=3$. If $M$ is now chosen large enough, every lattice point of the form $z_1 b_{m,1} + z_2 b_{m,2} + z_3 b_{m,3}$ with $z_1,z_2,z_3 \in \mathbb{Z}, z_3 \neq 0$ will be sufficiently far away from the plane spanned by $0$, $b_{m,1}$ and $b_{m,2}$ such that it will also lie outside of the ball.
Then $0$ and $b_{m,1} + m b_{m,2}$ are the only lattice points in the ball, and if they in fact lie on the boundary of the ball, it follows that $b_{m,1} + m b_{m,2}$ is a Voronoi-relevant vector, where the center of the ball serves as $x$ in Definition~\ref{defn:Voronoi_relevant} of Voronoi-relevant vectors.

\begin{figure}[tbhp]
\centering
\subfloat[$\mathcal{B}_{\Vert \cdot \Vert_3,1}(0)$.]
{\label{sfig:norm_front}
\includegraphics[width=0.23\textwidth]{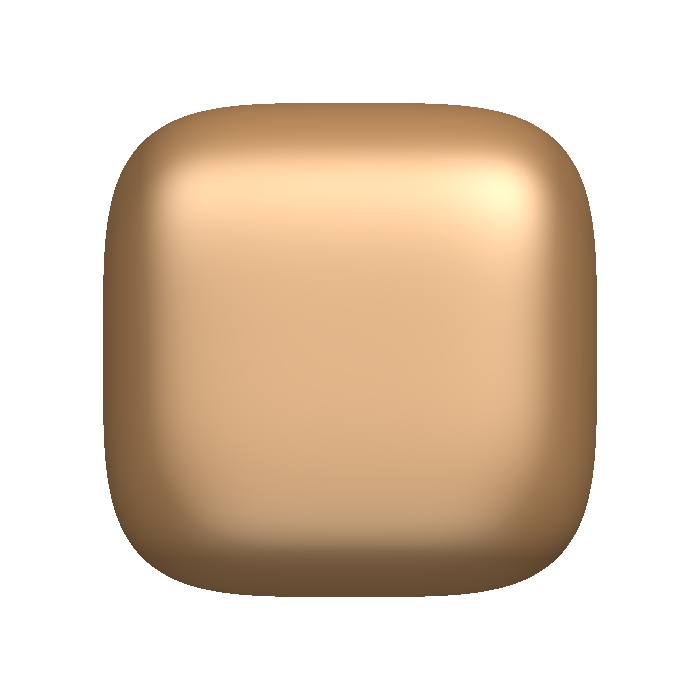}}
~
\subfloat[$\mathcal{B}_{\Vert \cdot \Vert_3,1}(0)$.]
{\label{sfig:norm_corner}
\includegraphics[width=0.23\textwidth]{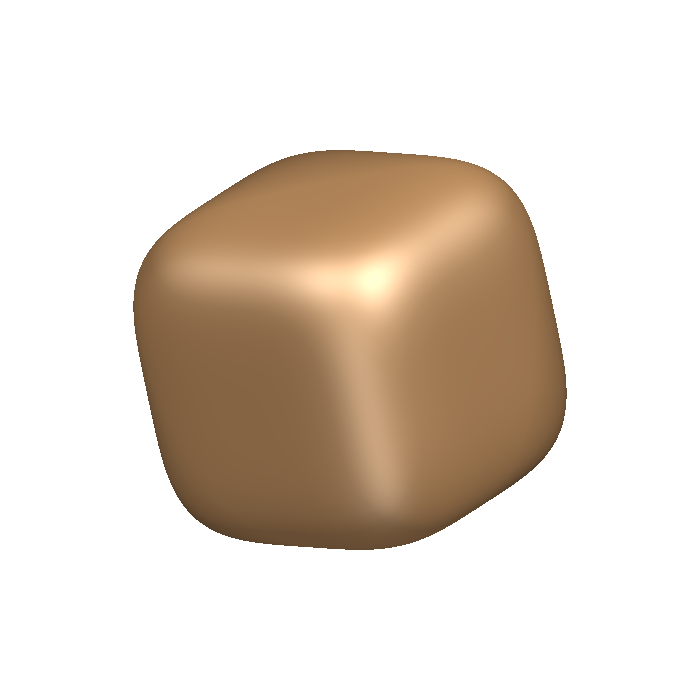}}
~
\subfloat[$\mathcal{B}_{\Vert \cdot \Vert_3,1}(0)$ and plane orthogonal to $z$-axis.]
{\label{sfig:norm_plane_ortho}
\includegraphics[width=0.23\textwidth]{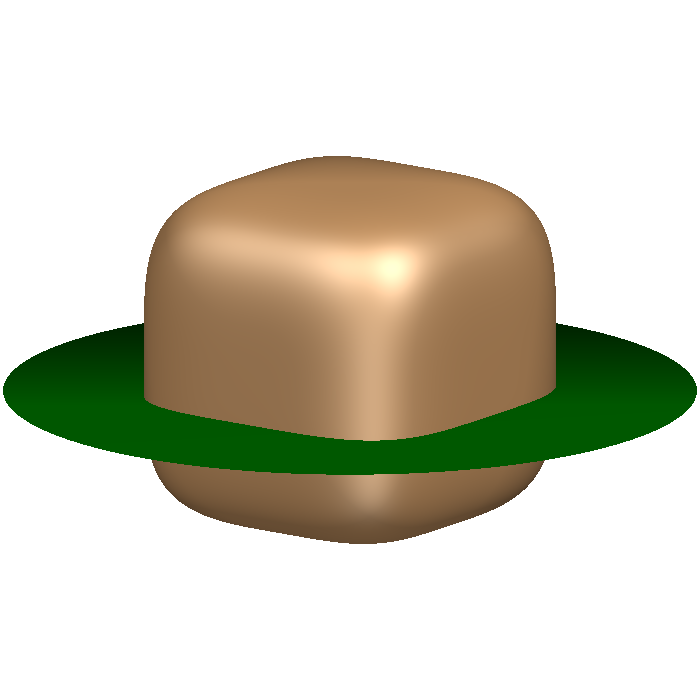}}
~
\subfloat[Figure~\ref{sfig:norm_plane_ortho} from perspective orthogonal to plane.]
{\label{sfig:norm_plane_ortho_intersect}
\includegraphics[width=0.23\textwidth]{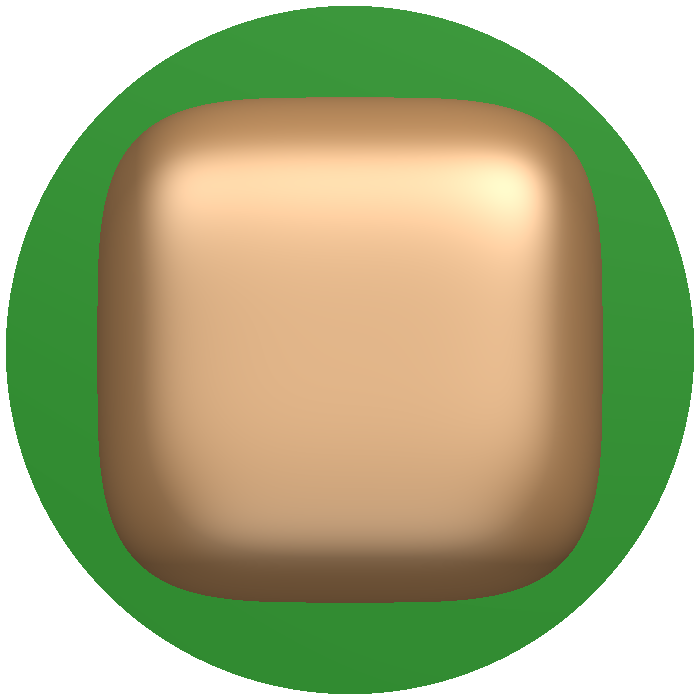}}
\\
\subfloat[$\mathcal{B}_{\Vert \cdot \Vert_3,1}(0)$ and plane intersecting the ball at an edge.]
{\label{sfig:norm_plane_45}
\includegraphics[width=0.23\textwidth]{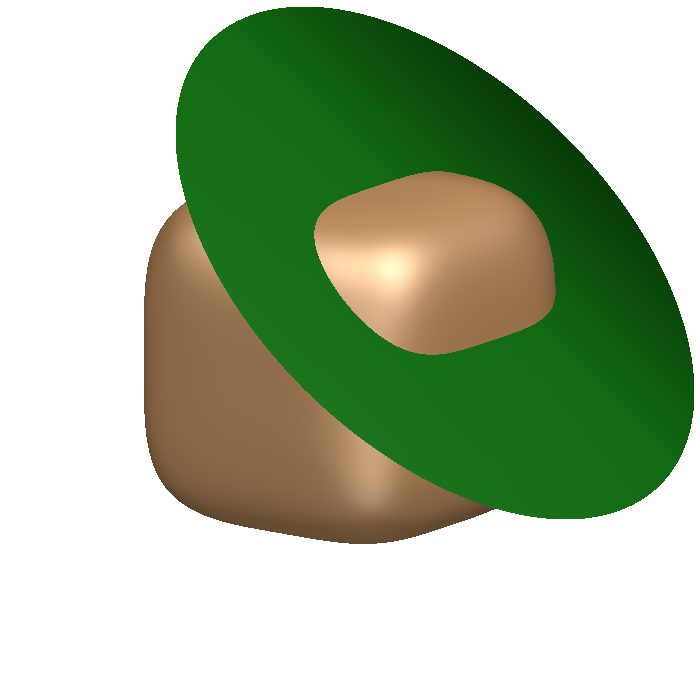}}
~
\subfloat[Figure~\ref{sfig:norm_plane_45} from perspective orthogonal to plane.]
{\label{sfig:norm_plane_45_intersect}
\includegraphics[width=0.23\textwidth]{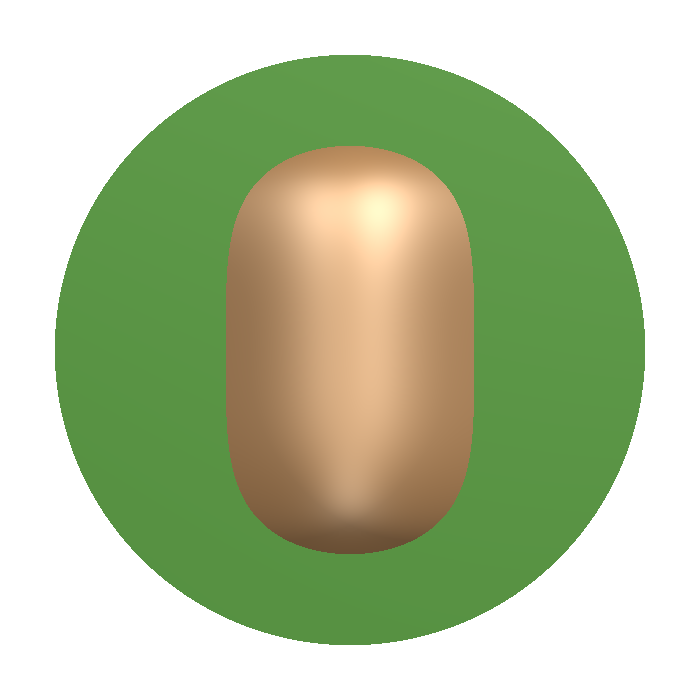}}
~
\subfloat[$\mathcal{B}_{\Vert \cdot \Vert_3,1}(0)$ and plane intersecting the ball at an edge.]
{\label{sfig:norm_plane_45extreme}
\includegraphics[width=0.23\textwidth]{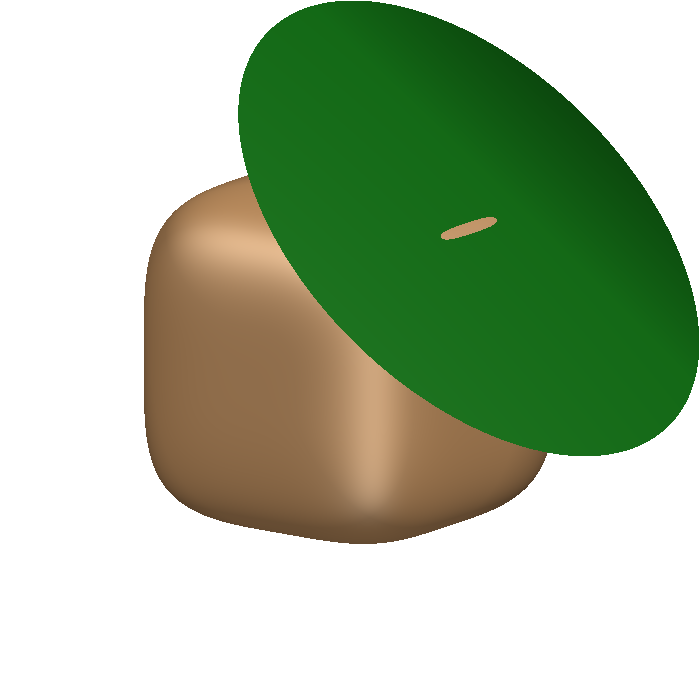}}
~
\subfloat[Figure~\ref{sfig:norm_plane_45extreme} from perspective orthogonal to plane.]
{\label{sfig:norm_plane_45extreme_intersect}
\includegraphics[width=0.23\textwidth]{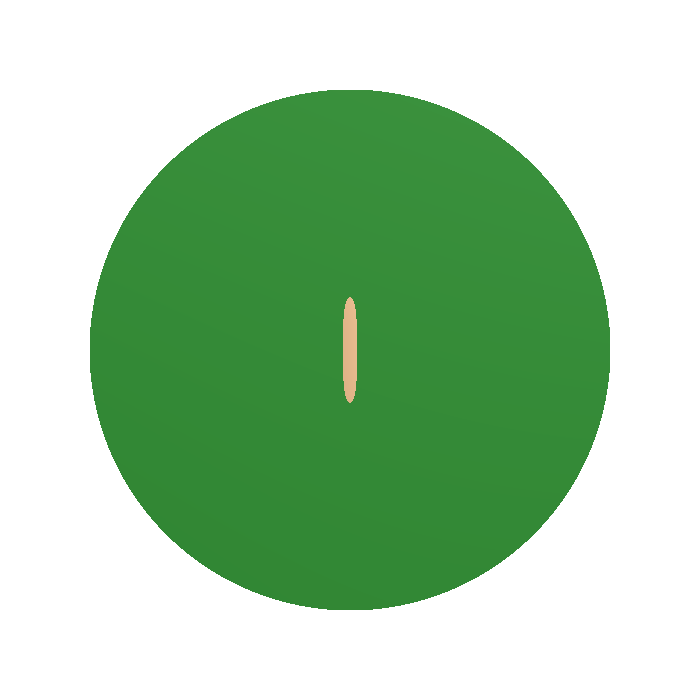}}
\caption{$\mathcal{B}_{\Vert \cdot \Vert_3,1}(0)$ intersecting different planes.}
\label{fig:3norm}
\end{figure}

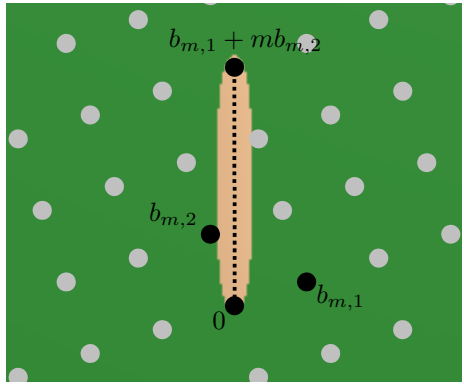
\begin{figure}[tbhp]
\centering
\begin{tikzpicture}
\node (1) at (0,0) {\includegraphics[scale=1]{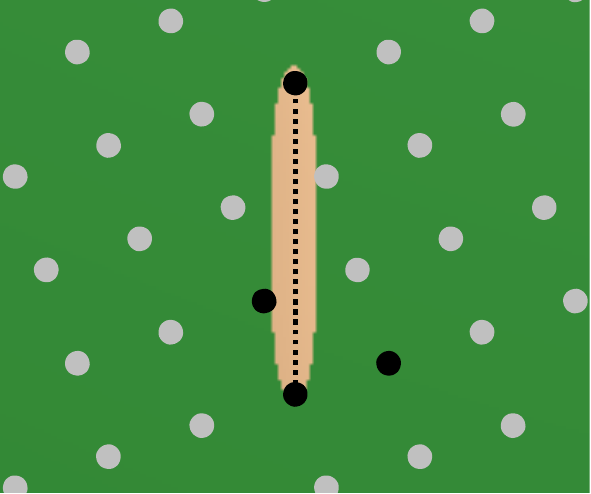}};
\node (2) at (-0.2,-1.7) {$0$};
\node (3) at (1.4,-1.4) {$b_{m,1}$};
\node (4) at (-0.8,-0.3) {$b_{m,2}$};
\node (5) at (0.15,2) {$b_{m,1}+m b_{m,2}$};
\end{tikzpicture}
\caption{Plane spanned by $0$, $b_{m,1}$ and $b_{m,2}$ intersects ball at an edge such that line segment between $0$ and $b_{m,1} + m b_{m,2}$ lies on the edge (cf. Figure~\ref{sfig:norm_plane_45extreme_intersect}).}
\label{fig:plane_through_ball}
\end{figure}

\begin{figure}[tbhp]
\centering
\begin{tikzpicture}
\node (1) at (0,0) {\includegraphics[width=0.9\textwidth]{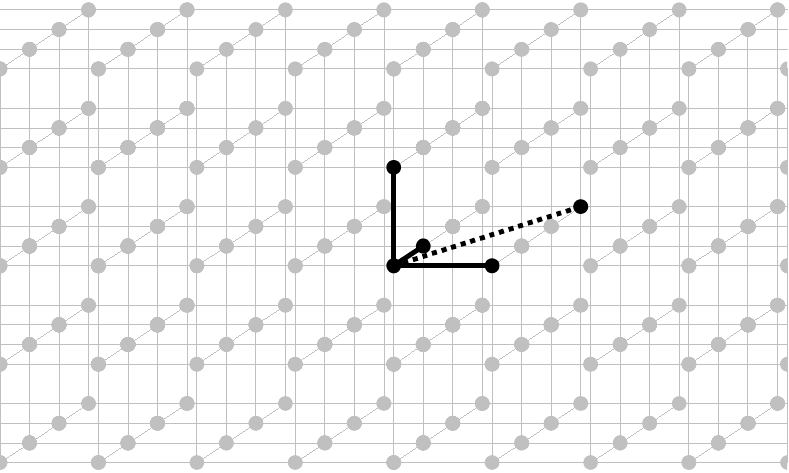}};
\node (2) at (-0.25,-0.8) {$0$};
\node (3) at (2.1,-0.8) {$e_1$};
\node (4) at (0.3,0.1) {$e_2$};
\node (5) at (-0.25,1) {$e_3$};
\node (6) at (4.3,0.8) {$e_1+me_2$};
\end{tikzpicture}
\caption{$\mathcal{L}(e_1, e_2, e_3)$.}
\label{fig:lattice1}
\end{figure}

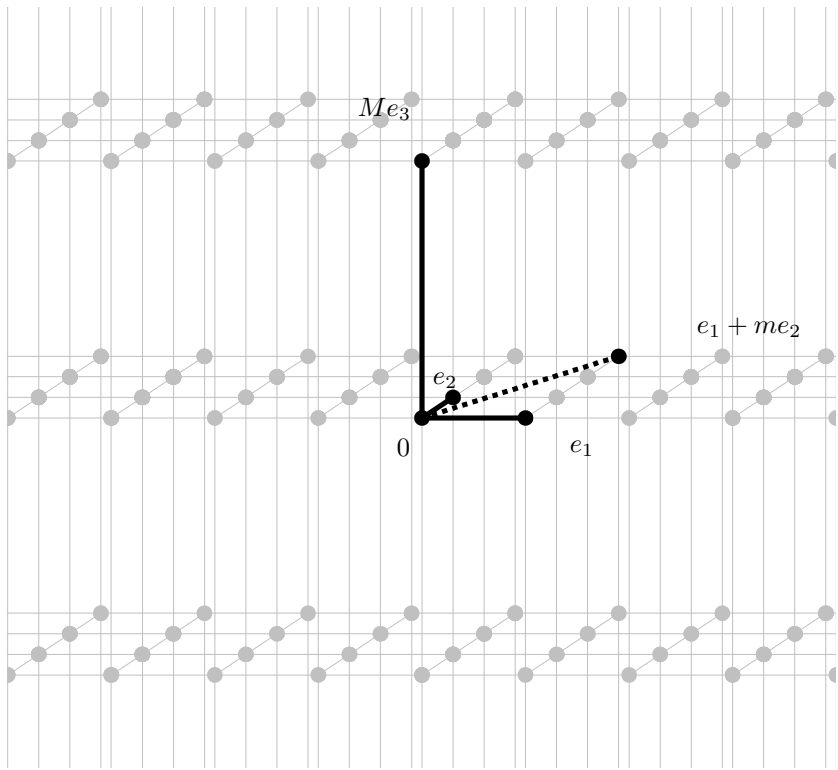
\begin{figure}[tbhp]
\centering
\begin{tikzpicture}
\node (1) at (0,0) {\includegraphics[width=0.9\textwidth]{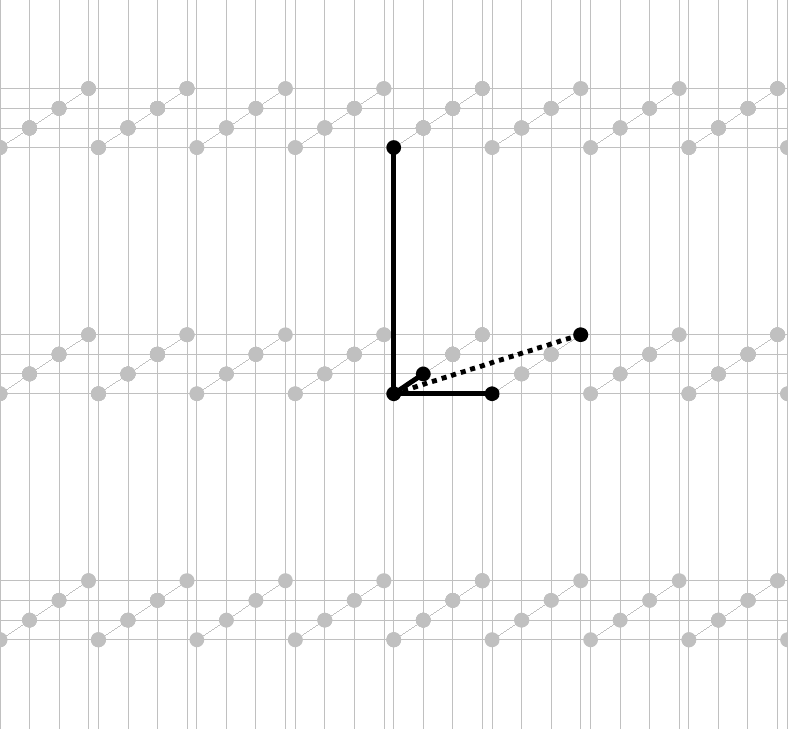}};
\node (2) at (-0.25,-0.8) {$0$};
\node (3) at (2.1,-0.8) {$e_1$};
\node (4) at (0.3,0.1) {$e_2$};
\node (5) at (-0.5,3.7) {$Me_3$};
\node (6) at (4.3,0.8) {$e_1+me_2$};
\end{tikzpicture}
\caption{$\mathcal{L}(e_1, e_2, M e_3)$:
Note that $M$ is so large that this figure is not scaled properly.}
\label{fig:lattice2}
\end{figure}

\begin{figure}[tbhp]
\centering
\begin{tikzpicture}
\node (1) at (0,0) {\includegraphics[width=0.9\textwidth]{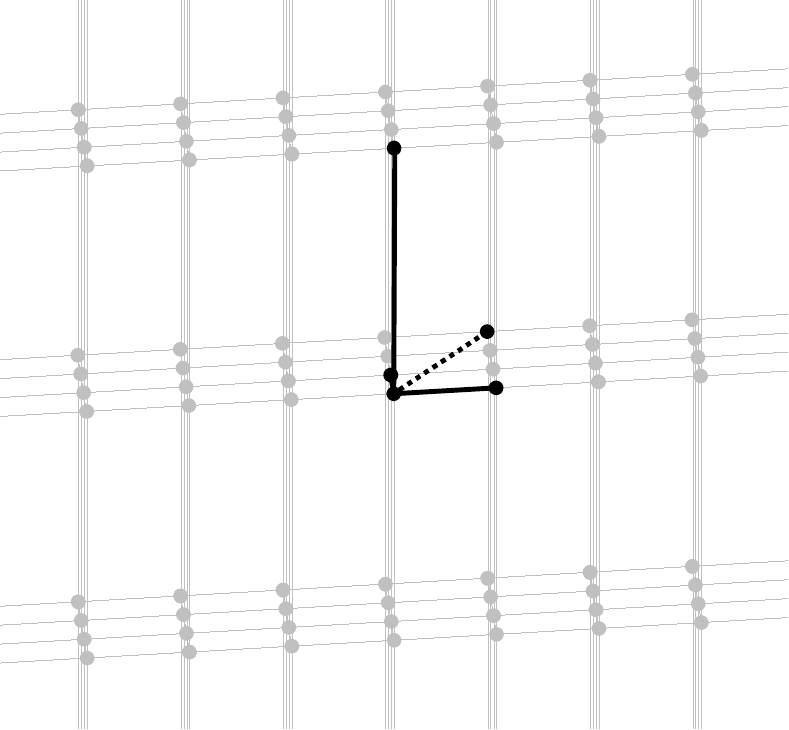}};
\node (2) at (-0.25,-0.8) {$0$};
\node (3) at (2.35,-0.7) {$R_ze_1$};
\node (4) at (-0.7,0) {$R_ze_2$};
\node (5) at (-0.5,3.7) {$Me_3$};
\node (6) at (2.9,0.9) {$R_z(e_1+me_2)$};
\end{tikzpicture}
\caption{$R_z \mathcal{L}(e_1, e_2, M e_3)$.}
\label{fig:lattice3}
\end{figure}

\begin{figure}[tbhp]
\centering
\begin{tikzpicture}
\node (1) at (0,0) {\includegraphics[width=0.9\textwidth]{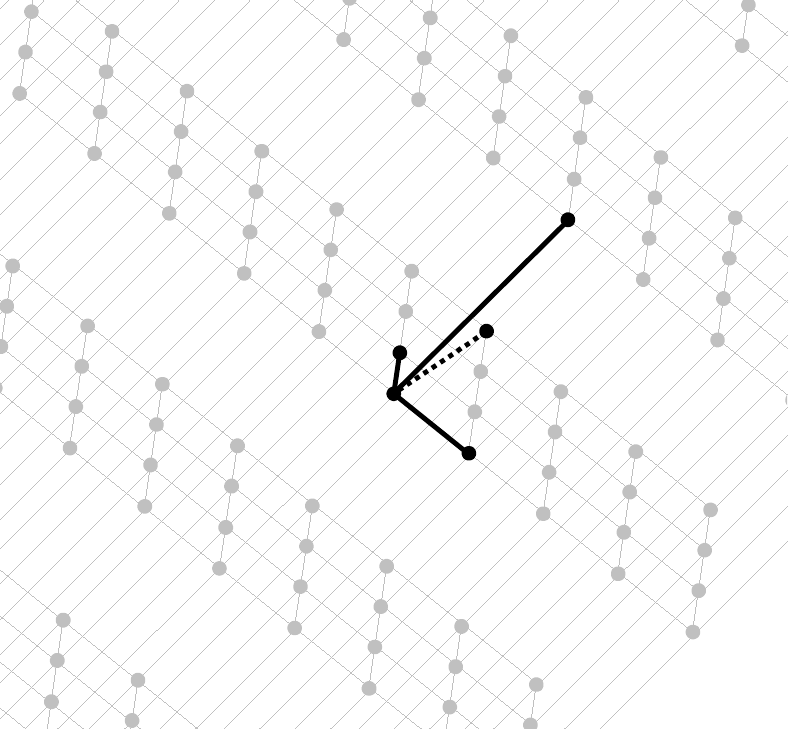}};
\node (2) at (-0.25,-0.8) {$0$};
\node (3) at (1.85,-1.95) {$b_{m,1}$};
\node (4) at (-0.4,0.3) {$b_{m,2}$};
\node (5) at (2.7,2.9) {$b_{m,3}$};
\node (6) at (2.9,0.9) {$b_{m,1}+mb_{m,2}$};
\end{tikzpicture}
\caption{$\mathcal{L}_m = R_y R_z \mathcal{L}(e_1, e_2, M e_3)$.}
\label{fig:lattice4}
\end{figure}

With these figurative ideas at hand, the rotations of $\mathcal{L}(e_1, e_2, M e_3)$ will now be described formally. These modifications of the standard lattice are also illustrated in Figures~\ref{fig:lattice1} to~\ref{fig:lattice4} for the case $m=3$.
First, $\mathcal{L}(e_1, e_2, M e_3)$ is rotated around the $z$-axis until $e_1 + m e_2$ lies on the $y$-axis, because all edges of the unit ball are parallel to the $x$-, $y$- or $z$-axis. This rotation is realized by the matrix
\begin{align*}
R_z := \begin{pmatrix}\frac{m}{\sqrt{{m}^{2}+1}} & -\frac{1}{\sqrt{{m}^{2}+1}} & 0\cr \frac{1}{\sqrt{{m}^{2}+1}} & \frac{m}{\sqrt{{m}^{2}+1}} & 0\cr 0 & 0 & 1\end{pmatrix}.
\end{align*}
Secondly, the resulting lattice $R_z \mathcal{L}(e_1, e_2, M e_3)$ is rotated around the $y$-axis by $45^\circ$ such that after the rotation the plane formerly spanned by $0$, $e_1$ and $e_2$ intersects translated unit balls of the $\ell_3$-norm at one of their edges. The second rotation is given by the matrix
\begin{align*}
R_y := \begin{pmatrix}\frac{1}{\sqrt{2}} & 0 & \frac{1}{\sqrt{2}}\cr 0 & 1 & 0\cr -\frac{1}{\sqrt{2}} & 0 & \frac{1}{\sqrt{2}}\end{pmatrix}.
\end{align*}

The resulting lattice $\mathcal{L}_m := R_y R_z \mathcal{L}(e_1, e_2, M e_3)$ is spanned by
\begin{align*}
b_{m,1} &:= R_y R_z e_1 = \frac{1}{\sqrt{{m}^{2}+1}} \begin{pmatrix}
\frac{m}{\sqrt{2}} \cr
1 \cr
-\frac{m}{\sqrt{2}}
\end{pmatrix}, \\
b_{m,2} &:= R_y R_z e_2 = \frac{1}{\sqrt{{m}^{2}+1}} \begin{pmatrix}
-\frac{1}{\sqrt{2}} \cr
m \cr
\frac{1}{\sqrt{2}} 
\end{pmatrix} \text{ and} \\
b_{m,3} &:= M R_y R_z e_3 = \frac{M}{\sqrt{2}} \begin{pmatrix}
1 \cr 
0 \cr
1
\end{pmatrix}.
\end{align*}
Using an appropriate scaling and translation of the unit ball, the situation in Figure~\ref{fig:plane_through_ball} can be reached. As already mentioned, this can be used to show that $b_{m,1}+m b_{m,2}$ is Voronoi-relevant if $M$ is sufficiently large.
For us it it is more important though that $\mathcal{L}_m$ has considerably more Voronoi-relevant vectors:
We show in the following Theorem~\ref{thm:main} that choosing $M:= 5 \sqrt{2} m^5$ implies that $b_{m,1}+k b_{m,2}$ is Voronoi-relevant with respect to $\Vert \cdot \Vert_3$ for every $k \in \mathbb{Z}, k \in [2,\sqrt{m}]$. 
Hence, we have for the $\ell_3$-norm that every $\mathcal{L}_m$ has $\Omega(\sqrt{m})$ Voronoi-relevant vectors, which implies Theorem~\ref{thm:general_dimensions}.

Note that our intuitive description above indicates that Theorem~\ref{thm:general_dimensions} holds for all $\ell_p$-norms with $p \in \mathbb{R}_{>2}$.
Our construction relies only on the fact that the unit ball of the $\ell_3$-norm is a cube with rounded edges and corners.
The $\ell_p$-unit balls interpolate between the Euclidean ball and the cube for $p \in (2, \infty)$.
Thus, our ideas apply to every $p$ in this range.
The only parameter one has to adapt to different $p$'s is the scaling factor $M$.
If $p$ grows, the $\ell_p$-unit ball tends to the cube and $M$ can be chosen smaller.
If $p$ converges to $2$, the $\ell_p$-unit ball starts to resemble  the Euclidean ball and $M$ tends to infinity.

\begin{theorem}
\label{thm:main}
For all $k,m \in \mathbb{Z}$ with $2 \leq k \leq \sqrt{m}$, we have that $b_{m,1}+k b_{m,2}$ is Voronoi-relevant in $\mathcal{L}_m$ with respect to $\Vert \cdot \Vert_3$.
\end{theorem}
In the proof of this statement, we need to calculate the distance between some $v \in \mathbb{R}^3$ and the plane spanned by $0$, $e_1$ and $e_2$ after this plane is translated along the $z$-axis and rotated around the $y$-axis by $45^\circ$.
\begin{lemma}
\label{lem:distance_to_plain}
For $C \in \mathbb{R}$ and $v = (\alpha, \beta, \gamma)^T \in \mathbb{R}^3$,
the unique point in $E_C := R_y ( \mathbb{R}e_1 + \mathbb{R} e_2 + C e_3 ) $ which is closest to $v$ with respect to $\Vert \cdot \Vert_3$ is $R_y (  \frac{\alpha-\gamma}{\sqrt{2}}, \beta, C )^T$, and  $\Vert v-E_C \Vert_3^3 = \frac{1}{4} |\sqrt{2}C-\alpha-\gamma|^3.$ 
\end{lemma}
\begin{proof}
 This follows from the fact that the function $f: \mathbb{R} \to \mathbb{R}_{\geq 0}, 
x \mapsto |C+x|^3 + |D-x|^3$ has a global minimum at $\frac{D-C}{2}$ with function value $\frac{1}{4}|D+C|^3$ if $C,D \in \mathbb{R}$ and $D \neq -C$.
\end{proof}
\begin{proofofmaintheorem}
For $k,m \in \mathbb{Z}$ with $m \geq k \geq 2$ define
\begin{align*}
x_{m,k} := \frac{1}{2} (b_{m,1} + k b_{m,2}) + \left( \begin{array}{c}
\left( \frac{k^2}{4} + \frac{1}{3} \right) m \\
0 \\
\left( \frac{k^2}{4} + \frac{1}{3} \right) m 
\end{array} \right)
= \left( \begin{array}{c}
\frac{m-k}{2\sqrt{2} \sqrt{m^2+1}} + \left( \frac{k^2}{4} + \frac{1}{3} \right) m \\
\frac{km+1}{2 \sqrt{m^2+1}} \\
\frac{k-m}{2\sqrt{2} \sqrt{m^2+1}} + \left( \frac{k^2}{4} + \frac{1}{3} \right) m
\end{array} \right).
\end{align*}
Then we have that $\Vert x_{m,k} \Vert_3 = \Vert b_{m,1} + k b_{m,2} - x_{m,k} \Vert_3$. 
To prove Theorem~\ref{thm:main}, we need to show $\Vert x_{m,k} \Vert_3 < \Vert x_{m,k} - v \Vert_3$ for all $v \in \mathcal{L}_m \setminus \lbrace 0, b_{m,1} + k b_{m,2} \rbrace$ and all $m \geq k^2$.
\begin{claim}
\label{claim:z3irrel}
$\Vert x_{m,k} \Vert_3 < \Vert z_1 b_{m,1} + z_2 b_{m,2} + z_3 b_{m,3} -x_{m,k} \Vert_3$ for all all $m \geq k$ and all $z_1, z_2, z_3 \in \mathbb{Z}, z_3 \neq 0$.
\end{claim}
\begin{proofofclaim}
Let $z_1, z_2, z_3 \in \mathbb{Z}$.
Since $R_z e_3 = e_3$ and $z_1 R_z e_1 + z_2 R_z e_2 \in \mathbb{R}e_1 + \mathbb{R}e_2$, we have that $z_1 b_{m,1} + z_2 b_{m,2} + z_3 b_{m,3} = R_y (z_1 R_z e_1 + z_2 R_z e_2 + M z_3 e_3) \in E_{M z_3}$.
Hence, it follows from  Lemma~\ref{lem:distance_to_plain} and $M= 5 \sqrt{2} m^5$ that
\begin{align*}
\Vert z_1 b_{m,1} + z_2 b_{m,2} + z_3 b_{m,3} -x_{m,k} \Vert_3^3 
&\geq \Vert x_{m,k} - E_{M z_3} \Vert_3^3 \\
&= \frac{1}{4} \left|\sqrt{2} M z_3 - 2 \left( \frac{k^2}{4}+\frac{1}{3} \right) m \right|^3 \\
&= \frac{1}{4} \cdot 1000 m^{15} \left| z_3 - \frac{1}{5 m^4} \left(\frac{k^2}{4} +\frac{1}{3} \right) \right|^3.
\end{align*}
The prerequisite $m \geq k \geq 2$ yields $\frac{1}{5 m^4} \left(\frac{k^2}{4} +\frac{1}{3} \right) \in \left( 0, \frac{1}{60} \right]$.
Thus, for $z_3 \in \mathbb{Z} \setminus \lbrace 0 \rbrace$, the inequality $z_3-\frac{1}{5 m^4} \left(\frac{k^2}{4} +\frac{1}{3} \right) \geq 0$ is equivalent to $z_3 \geq 1$,
 and $\left| z_3 - \frac{1}{5 m^4} \left(\frac{k^2}{4} +\frac{1}{3} \right) \right|$ is minimized for $z_3=1$.
 This shows
\begin{align*}
\Vert z_1 b_{m,1} + z_2 b_{m,2} + z_3 b_{m,3} -x_{m,k} \Vert_3^3 
&\geq 250 m^{15} \left( 1 - \frac{1}{5 m^4} \left(\frac{k^2}{4} +\frac{1}{3} \right) \right)^3
> 200 m^{15}.
\end{align*}
The desired inequality follows from $\Vert x_{m,k} \Vert_3^3 < 4 m^{15}$. 
\end{proofofclaim}
Due to Claim~\ref{claim:z3irrel}, it is left to show that the restriction of
\begin{align*}
f: \mathbb{R} \times \mathbb{R} &\longrightarrow \mathbb{R}_{\geq 0}, \\
(r_1, r_2) &\longmapsto \Vert r_1 b_{m,1} + r_2 b_{m,2}  -x_{m,k} \Vert_3^3
\end{align*}
to $\mathbb{Z} \times \mathbb{Z}$ is globally minimized at $(0,0)$ and $(1,k)$ if $m \geq k^2$.
The global minimum of $f$ over $\mathbb{R} \times \mathbb{R}$ is achieved at $(\frac12, \frac{k}{2})$, which follows from Lemma~\ref{lem:distance_to_plain} since $r_1 b_{m,1} + r_2 b_{m,2} \in E_0$ for all $r_1,r_2 \in \mathbb{R}$.
Moreover, $f$ is symmetric about $(\frac12, \frac{k}{2})$, i.e., $f(r_1,r_2) = f(1-r_1,k-r_2)$.
The latter can be seen, using the abbreviation $\alpha_{m,k} := ( \frac{k^2}{4} + \frac{1}{3} ) m$, as follows:
\begin{align*}
  f(1-r_1,k-r_2)
  &= \left\Vert \left( \frac12-r_1 \right) b_{m,1} + \left( \frac{k}{2}-r_2 \right) b_{m,2} - (\alpha_{m,k},0,\alpha_{m,k})^T \right\Vert_3^3\\
  &= \left\Vert \left( r_1-\frac12 \right) b_{m,1} + \left( r_2-\frac{k}{2} \right) b_{m,2} - (\alpha_{m,k},0,\alpha_{m,k})^T \right\Vert_3^3\\
  &= f(r_1,r_2),
\end{align*}
where the middle equation holds since $( \frac12-r_1 ) b_{m,1} + ( \frac{k}{2}-r_2) b_{m,2}$ is of the form $( \beta_1, \beta_2, - \beta_1)^T$.
By this symmetry, it is enough to show that $f$ restricted to $\mathbb{Z}_{\leq 0} \times \mathbb{Z}$ has a unique minimum at $(0,0)$.
We complete this proof by comparing $f(0,0)$ with the values of $f$ restricted to $\lbrace 0 \rbrace \times \mathbb{Z}$ in Claim~\ref{claim:z1is0} and then with the values of $f$ restricted to $\mathbb{Z}_{< 0} \times \mathbb{Z}$ in Claim~\ref{claim:z1smaller0}.
\begin{claim}
\label{claim:z1is0}
$f(0,0) < f(0,z_2)$ for all $z_2 \in \mathbb{Z} \setminus \lbrace 0 \rbrace$.
\end{claim}
\begin{proofofclaim}
The function $f$ is strictly convex since $x \mapsto \Vert x \Vert_3^3$ is a strictly convex function. 
  Thus, it is enough to show $f(0,0)<f(0,1)$ and $f(0,0)<f(0,-1)$.
  Due to $\sqrt{m} \geq k \geq 2$ we have that 
\begin{align*}
  f(0,0)
&= \left(\frac{m-k}{2\sqrt{2} \sqrt{m^2+1}} + \left( \frac{k^2}{4} + \frac{1}{3} \right) m \right)^3 + \left( \frac{k-m}{2\sqrt{2} \sqrt{m^2+1}} + \left( \frac{k^2}{4} + \frac{1}{3} \right) m \right)^3 \\
&\,\quad\, + \left(\frac{km+1}{2 \sqrt{m^2+1}} \right)^3 \\
&= 2m^3 \left( \frac{k^2}{4} + \frac{1}{3} \right)^3  + 6m \left( \frac{k^2}{4} + \frac{1}{3} \right) \left( \frac{m-k}{2\sqrt{2} \sqrt{m^2+1}} \right)^2  + \left(\frac{km+1}{2 \sqrt{m^2+1}} \right)^3,
\end{align*}
and analogously
\begin{align*}
f(0,1)  
&= 2m^3 \left( \frac{k^2}{4} + \frac{1}{3} \right)^3 \hspace*{-1mm}  + 6m \left( \frac{k^2}{4} + \frac{1}{3} \right) \left( \frac{m-k+2}{2\sqrt{2} \sqrt{m^2+1}} \right)^2 \hspace*{-1mm}+ \left(\frac{(k-2)m+1}{2 \sqrt{m^2+1}} \right)^3,\\
f(0,-1)
&= 2m^3 \left( \frac{k^2}{4} + \frac{1}{3} \right)^3 \hspace*{-1mm} + 6m \left( \frac{k^2}{4} + \frac{1}{3} \right) \left( \frac{m-k-2}{2\sqrt{2} \sqrt{m^2+1}} \right)^2 \hspace*{-1mm} + \left(\frac{(k+2)m+1}{2 \sqrt{m^2+1}} \right)^3.
\end{align*}
First, we use $\sqrt{m^2+1} > m$ and $m \geq k^2$ to derive
\begin{align*}
  \frac{8\sqrt{m^2+1}^3}{m} \left( f(0,1)-f(0,0) \right)
  &= 24\left( \frac{k^2}{4} + \frac{1}{3} \right)\sqrt{m^2+1}(m-k+1) \\ &\,\quad\, + m^2(-6k^2+12k-8)+m(-12k+12)-6\\
  &> 12 m^2 k + m(-6k^3+6k^2-20k+20)-6 \\
  &\geq 6mk^3+6mk^2-20mk+20m-6,
\end{align*}
which is positive since $20m > 6$ and $6mk^3 > 20mk$ due to $k \geq 2$. This shows $f(0,0)<f(0,1)$.
Secondly, we consider
\begin{align}
  \label{eq:f01}
  \begin{split}
  \frac{8\sqrt{m^2+1}^3}{m} \left( f(0,-1)-f(0,0) \right)
  &= \sqrt{m^2+1} \left(6k^3+6k^2+8k+8 - m(6k^2+8) \right)
  \\&\,\quad\, + m^2(6k^2+12k+8)+m(12k+12)+6.
\end{split}
\end{align}
Since $m \geq k$, we have that $m^2 \left(  (6k^2+12k+8)^2 -  (6k^2+8)^2  \right) > (6k^2+8)^2$, which is equivalent to
$m(6k^2+12k+8) > \sqrt{m^2+1}(6k^2+8)$.
Hence,~\eqref{eq:f01} is positive and $f(0,0)<f(0,-1)$.
\end{proofofclaim}
\begin{claim}
\label{claim:z1smaller0}
$f(0,0) < f(z_1,z_2)$ for all $z_1,z_2 \in \mathbb{Z}, z_1 < 0$.
\end{claim}
\begin{proofofclaim}
  We have that
\begin{align*}
  f(z_1,z_2)
&= \left| \left(\frac{k^2}{4}+\frac{1}{3} \right) m - \frac{(2 z_1 - 1)m - (2 z_2-k)}{2\sqrt{2} \sqrt{m^2+1}} \right|^3 \\
&\,\quad\, + \left| \left(\frac{k^2}{4}+\frac{1}{3} \right) m + \frac{(2 z_1 - 1)m - (2 z_2-k)}{2\sqrt{2} \sqrt{m^2+1}} \right|^3 \\
&\,\quad\,  + \left| \frac{(2 z_2-k)m + (2 z_1-1)}{2 \sqrt{m^2+1}} \right|^3.
\end{align*}
If $\left(\frac{k^2}{4}+\frac{1}{3} \right) m - \frac{(2 z_1 - 1)m - (2 z_2-k)}{2\sqrt{2} \sqrt{m^2+1}} < 0$ or $\left(\frac{k^2}{4}+\frac{1}{3} \right) m + \frac{(2 z_1 - 1)m - (2 z_2-k)}{2\sqrt{2} \sqrt{m^2+1}}<0$, then $f(z_1,z_2) > 8 \left( \frac{k^2}{4} + \frac{1}{3} \right)^3m^3$. Assume in this case for contradiction that $f(z_1,z_2) \leq f(0,0)$. This implies 
$6 \left( \frac{k^2}{4} + \frac{1}{3} \right)^3m^3 < 6 \left( \frac{k^2}{4} + \frac{1}{3} \right) m \left( \frac{m-k}{2\sqrt{2} \sqrt{m^2+1}} \right)^2  + \left(\frac{km+1}{2 \sqrt{m^2+1}} \right)^3$. Dividing by $6 \left( \frac{k^2}{4} + \frac{1}{3} \right) m$ and multiplying by $8 \sqrt{m^2+1}^3$ yields
\begin{align*}
8 \left( \frac{k^2}{4} + \frac{1}{3} \right)^2 m^2 \sqrt{m^2+1}^3 < (m-k)^2\sqrt{m^2+1} + \frac{2 (km+1)^3}{m (3k^2+4)}.
\end{align*}
Using $m<\sqrt{m^2+1}<\sqrt{2}m$ and $2 (km+1) < m (3 k^2+4)$ leads to
$8 \frac{k^4}{16}m^5 < \sqrt{2}m^3 + (km+1)^2 < \sqrt{2}m^3 + 4k^2m^2$.
Hence $16 m^3 \leq k^4 m^3 < 2 \sqrt{2} m + 8 k^2 <
\linebreak[4]
 4 m + 8 m^2$ follows, leading to
$0 > 4 m^2 - 2m-1 = 4 \left( m - \frac{1+\sqrt{5}}{4}\right) \left( m - \frac{1-\sqrt{5}}{4} \right)$, but this is a contradiction since $\frac{1-\sqrt{5}}{4} < \frac{1+\sqrt{5}}{4}<1$ and $m \geq 2$. This shows $f(z_1,z_2) > f(0,0)$ in case that $\left(\frac{k^2}{4}+\frac{1}{3} \right) m - \frac{(2 z_1 - 1)m - (2 z_2-k)}{2\sqrt{2} \sqrt{m^2+1}} < 0$ or 
\linebreak[4]
$\left(\frac{k^2}{4}+\frac{1}{3} \right) m + \frac{(2 z_1 - 1)m - (2 z_2-k)}{2\sqrt{2} \sqrt{m^2+1}}<0$.

Hence it can be assumed in the following that 
\begin{align*}
\begin{split}
f(z_1,z_2) 
&= \left( \left(\frac{k^2}{4}+\frac{1}{3} \right) m - \frac{(2 z_1 - 1)m - (2 z_2-k)}{2\sqrt{2} \sqrt{m^2+1}} \right)^3 \\
&\,\quad\,
 + \left( \left(\frac{k^2}{4}+\frac{1}{3} \right) m + \frac{(2 z_1 - 1)m - (2 z_2-k)}{2\sqrt{2} \sqrt{m^2+1}} \right)^3  \\
&\,\quad\, + \left| \frac{(2 z_2-k)m + (2 z_1-1)}{2 \sqrt{m^2+1}} \right|^3 \\
&=\left| \frac{(2 z_2-k)m + (2 z_1-1)}{2 \sqrt{m^2+1}} \right|^3 + 2m^3 \left(\frac{k^2}{4}+\frac{1}{3} \right)^3 \\
&\,\quad\, + 6m \left(\frac{k^2}{4}+\frac{1}{3} \right) \left( \frac{(2 z_1 - 1)m - (2 z_2-k)}{2\sqrt{2} \sqrt{m^2+1}} \right)^2.
\end{split}
\end{align*}
Thus, $f(z_1,z_2) > f(0,0)$ is equivalent to
\begin{align*}
& 6m \left(\frac{k^2}{4}+\frac{1}{3} \right) \left( \frac{(2 z_1 - 1)m - (2 z_2-k)}{2\sqrt{2} \sqrt{m^2+1}} \right)^2 + \left| \frac{(2 z_2-k)m + (2 z_1-1)}{2 \sqrt{m^2+1}} \right|^3\\
>\,& 6m \left( \frac{k^2}{4} + \frac{1}{3} \right) \left( \frac{m-k}{2\sqrt{2} \sqrt{m^2+1}} \right)^2  + \left(\frac{km+1}{2 \sqrt{m^2+1}} \right)^3
\end{align*}
and consequently to 
\begin{align*}
g(z_1,z_2)  &:= 6m \left( \frac{k^2}{4} + \frac{1}{3} \right) \sqrt{m^2+1} \left( \left( (2 z_1 - 1)m - (2 z_2-k) \right)^2 - (m-k)^2 \right) \\ 
&\,\quad\, + |(2 z_2-k)m + (2 z_1-1)|^3- (km+1)^3 \\
&>0.
\end{align*}
If $-(2z_1-1)m+(2z_2-k) > m-k$ and $-(2z_2-k)m-(2z_1-1) > km+1$, then $g(z_1,z_2)$ is clearly positive and we are done.
If $-(2z_1-1)m+(2z_2-k) \leq m-k$, then $z_1 \leq -1$ implies $2z_2-k \leq -2m-k$ and we get
\begin{align*}
    g(z_1,z_2)
    &\geq - 6 \left( \frac{k^2}{4} + \frac{1}{3} \right) m \sqrt{m^2+1} (m-k)^2 \\
    &\,\quad\, + ( ( 2m+k)m+3)^3 - (km+1)^3 \\
    &\geq -12 \left( \frac{k^2}{4} + \frac{1}{3} \right) m^4 + ( (km+1)+2(m^2+1))^3 \\
    &\,\quad\,-(km+1)^3 \\
    &= -3k^2m^4-4m^4 + 6 (km+1)^2(m^2+1) \\
    &\,\quad\, + 12(km+1)(m^2+1)^2+8(m^2+1)^3\\
    &\geq -3k^2m^4 -4m^4 + 6k^2m^4 + 12 k m^5+8m^6 \\
    &>0.
\end{align*}
Finally, if $-(2z_2-k)m-(2z_1-1) \leq km+1$, then $z_1 \leq -1$ yields $-(2z_2-k) \leq k - \frac{2}{m}<k$. This implies $z_2 \geq 1$ and $2z_2-k \geq 2-k$, which leads to
\begin{align*}
    g(z_1,z_2) 
    &\geq 6 \left( \frac{k^2}{4} + \frac{1}{3} \right) m \sqrt{m^2+1} \left( (3m-k+2)^2 - (m-k)^2\right) \\
    &\,\quad\, - (km+1)^3\\
    &\geq 6 \frac{k^2}{4}m^2 \left( ( (m-k)+2(m+1))^2-(m-k)^2 \right) \\
    &\,\quad\,-(km+1)^3 \\
    &=6 k^2m^2 ( (m-k)(m+1)+(m+1)^2)-(km+1)^3 \\
    &\geq 12k^2m^4 +12 k^2m^3 + 3k^2m^2 - 7k^3m^3 - 3km - 1 \\
    &>0.
\end{align*}
Hence for $z_1 \leq -1$, we have shown that $g(z_1,z_2)$ is positive and that $f(z_1,z_2)>f(0,0)$.
\end{proofofclaim}
This concludes the proof of Theorem~\ref{thm:main}.
\end{proofofmaintheorem}

\section{Structural Properties of Voronoi Cells}
\label{sec:geom}
This section is devoted to show that the (weak) Voronoi-relevant vectors define the Voronoi cell of a lattice and to study how these vectors correspond to the facets of Voronoi cells.
In Subsection~\ref{ssec:genNorms} we discuss these questions for arbitrary norms and in Subsection~\ref{ssec:strNorms} for strictly convex norms.
Finally, we investigate the number of (weak) Voronoi-relevant vectors in two-dimensional lattices in Subsection~\ref{ssec:2Dim}.
First, we introduce some preliminary results.

For a given norm $\Vert \cdot \Vert$, the closed unit ball $\lbrace x \in \mathbb{R}^n \mid \Vert x \Vert \leq 1 \rbrace$ is a \emph{convex body}, i.e., a compact and convex subset of $\mathbb{R}^n$ with $0$ in its interior.
A convex body $K$ is called \emph{strictly convex} if for every $x,y \in K$ with $x \neq y$ and $\tau \in (0,1)$ it holds that $\tau x + (1-\tau)y$ lies in the interior of $K$.
A norm is strictly convex if and only if its closed unit ball is strictly convex. 
Every boundary point $p$ of a convex body $K$ has a supporting hyperplane, i.e., a hyperplane $H$ going through $p$ such that $K$ is contained in one of the two closed halfspaces bounded by $H$. A convex body is called \emph{smooth} if each point on its boundary has a unique supporting hyperplane. A norm is said to be \emph{smooth} if its closed unit ball is smooth, although such a norm is generally not smooth as a function.
For $p \in \mathbb{R}_{>1}$, the $\ell_p$-norms are examples of strictly convex and smooth norms. The $\ell_1$-norm and the $\ell_\infty$-norm have neither of the two properties.
Throughout this section, we use properties of the \emph{convex dual} of a given convex body $K$. This is defined as $K^\circ := \lbrace y \in \mathbb{R}^n \mid \forall x \in K: \langle x,y \rangle \leq 1 \rbrace$, where $\langle \cdot, \cdot \rangle$ denotes the Euclidean inner product on $\mathbb{R}^n$. 
The \emph{support function} of a convex body $K$ is defined by $h_K(u) := \sup \lbrace \langle x,u \rangle \mid x \in K \rbrace$ for $u \in \mathbb{R}^n$. The connection between these notions is shown in the following:
\begin{proposition}
\label{prop:duality}
For a convex body $K \subseteq \mathbb{R}^n$, the following assertions hold:
\begin{enumerate}
\item $K^\circ$ is a convex body with $(K^\circ)^\circ = K$.
\item If $K$ is the closed unit ball of a norm $\Vert \cdot \Vert$, then it holds for every $x \in \mathbb{R}^n$ that $\Vert x \Vert = h_{K^\circ}(x)$.
\item $K$ is strictly convex if and only if $K^\circ$ is smooth.
\end{enumerate}
\end{proposition}
\begin{proof}
  The first two assertions can be found in Theorem 14.5 and Corollary 14.5.1 of~\cite{rockafellar}. For the third assertion note first that $K$ is strictly convex if and only if supporting hyperplanes at distinct boundary points of $K$ are distinct. Hence the following are equivalent:
  \begin{itemize}
    \item $K$ is not strictly convex.
    \item There are $b_1, b_2 \in \partial K, b_1 \neq b_2$, with a common supporting hyperplane $\lbrace x \in \mathbb{R}^n \mid \langle x,a \rangle = 1 \rbrace$ such that $\langle x,a \rangle \leq 1$ for all $x \in K$. Note that $a \in \partial K^\circ$.
    \item There is $a \in \partial K^\circ$ with two distinct supporting hyperplanes $\lbrace y \in \mathbb{R}^n \mid \langle y,b_1 \rangle = 1 \rbrace$ and $\lbrace y \in \mathbb{R}^n \mid \langle y,b_2 \rangle = 1 \rbrace$ such that $\langle y,b_i \rangle \leq 1$ for all $y \in K^\circ$ and $i \in \lbrace 1,2 \rbrace$. Note that $b_1, b_2 \in \partial K$.
    \item $K^\circ$ is not smooth.
  \end{itemize}
\end{proof} 
In our geometric analysis of Voronoi cells, we will in particular study bisectors and their corresponding strict and non-strict \emph{halfspaces}: $\mathcal{H}_{\Vert \cdot \Vert}^<(a,b) := \lbrace x \in \mathbb{R}^n \mid \Vert x-a \Vert < \Vert x-b \Vert \rbrace$ and $\mathcal{H}_{\Vert \cdot \Vert}^\leq(a,b) := \mathcal{H}_{\Vert \cdot \Vert}^<(a,b) \cup \mathcal{H}_{\Vert \cdot \Vert}^=(a,b)$.
The following lemma implies that the Voronoi cell $\mathcal{V}(\Lambda, \Vert \cdot \Vert)$ and its variants $\tilde{\mathcal{V}}$ and $\tilde{\mathcal{V}}^{(g)}$ (see Theorems~\ref{thm:Voronoi-relevant_vectors_define_Voronoi-cell} and~\ref{thm:generalized_Voronoi-relevant_vectors_define_Voronoi-cell}) are star-shaped with the origin as center.
Recall that a subset $S$ of $\mathbb{R}^n$ is \emph{star-shaped} with center $c \in S$ if for all $x \in S$ the line segment from $c$ to $x$ is in $S$. 

\begin{lemma}
\label{lem:star}
Let $\Vert \cdot \Vert: \mathbb{R}^n \to \mathbb{R}_{\geq 0}$ be a norm and let $v \in \mathbb{R}^n \setminus \lbrace 0 \rbrace$.
\begin{enumerate}
\item The halfspaces $\mathcal{H}^{\leq}_{\Vert \cdot \Vert}(0,v)$ and $\mathcal{H}^{<}_{\Vert \cdot \Vert}(0,v)$ are star-shaped with center $0$.
\item Moreover, if $\Vert \cdot \Vert$ is strictly convex and $x \in \mathcal{H}^{=}_{\Vert \cdot \Vert}(0,v)$, then
\begin{align*}
&\tau \Vert x \Vert < \Vert \tau x - v \Vert \text{ for } \tau \in (0,1) \text{, and}\\
&\tau \Vert x \Vert > \Vert \tau x - v \Vert \text{ for } \tau > 1.
\end{align*}
\end{enumerate}
\end{lemma}

\begin{proof}
  \begin{enumerate}
    \item For $x \in \mathcal{H}^{\leq}_{\Vert \cdot \Vert}(0,v)$ and $\tau \in (0,1)$, we want to show that $\tau x \in \mathcal{H}^{\leq}_{\Vert \cdot \Vert}(0,v)$. This follows directly from
      \begin{align}
        \label{eq:star}
        \begin{split}
        \tau \Vert x \Vert &\leq \tau \Vert x-v \Vert
        = \Vert \tau (\tau x -v) + (1-\tau)\tau x \Vert\\
        &\leq \tau \Vert \tau x -v \Vert + (1-\tau) \tau \Vert x \Vert.
        \end{split}
      \end{align}
      If $x \in \mathcal{H}^{<}_{\Vert \cdot \Vert}(0,v)$, the first inequality in~\eqref{eq:star} becomes strict and $\tau x \in \mathcal{H}^{<}_{\Vert \cdot \Vert}(0,v)$.
    \item By strict convexity we have for $\tau \in (0,1)$ that
      \begin{align*}
        \tau \Vert x \Vert &= \tau \Vert x-v \Vert
        = \Vert \tau (\tau x -v) + (1-\tau)\tau x \Vert\\
        &< \max \lbrace \Vert \tau x -v \Vert, \tau \Vert x \Vert \rbrace
        = \Vert \tau x -v \Vert,
      \end{align*}
      and for $\tau >1$ that $\Vert \tau x -v \Vert = \Vert (1-\frac{1}{\tau}) \tau x + \frac{1}{\tau} \tau (x-v) \Vert < \tau \Vert x \Vert$.
  \end{enumerate}
\end{proof}

\subsection{General Norms}
\label{ssec:genNorms}
We will see in Subsection~\ref{ssec:strNorms} that strictly convex norms behave like the Euclidean norm in the two-dimensional setting (Theorems~\ref{thm:Voronoi-relevant_vectors_define_Voronoi-cell} and~\ref{thm:bijection_facets_Voronoi-relevant}). This is not true for other norms.
For non-strictly convex norms, the Voronoi-relevant vectors are in general not sufficient to determine the Voronoi cell of a two-dimensional lattice completely.
To see this, consider the lattice $\mathcal{L}_3$ from the proof of Proposition~\ref{prop:2dimensions_1norm} (i.e., the lattice $\mathcal{L}(b_1,3b_2)$ for $b_1 := (1,1)^T$ and $b_2 := (0,1)^T$) together with the $\ell_1$-norm $\Vert \cdot \Vert_1$. The Voronoi cell $\mathcal{V} (\mathcal{L}(b_1,3b_2),\Vert \cdot \Vert_1)$ is depicted in Figure~\ref{fig:Voronoi_cell_1norm}, and the only Voronoi-relevant vectors are $\pm b_1$. This shows that Theorem~\ref{thm:Voronoi-relevant_vectors_define_Voronoi-cell} does not hold for general non-strictly convex norms because~-- for example~-- we have that $ \frac{7}{8}(1,-1)^T \notin \mathcal{V} (\mathcal{L}(b_1,3b_2),\Vert \cdot \Vert_1)$ is closer to $0$ than to both Voronoi-relevant vectors with respect to the $\ell_1$-norm. Therefore we need a larger set of vectors for a description of the Voronoi cell.
The weak Voronoi-relevant vectors give such a description, for every norm and every lattice dimension.
\begin{figure}[tbhp]
  \centering
  \includegraphics[scale=0.5]{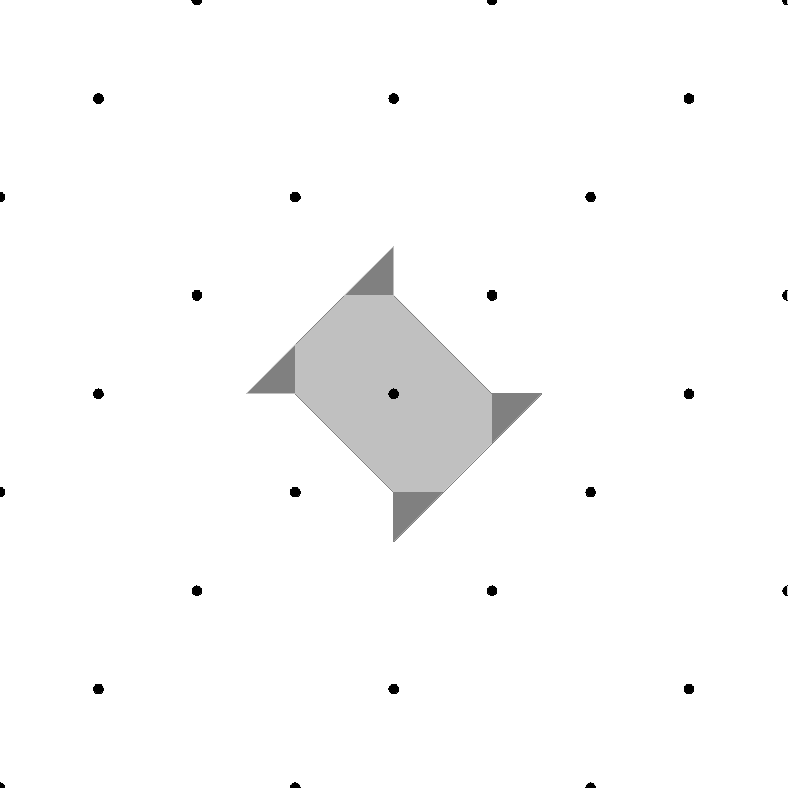}
  \caption{$\mathcal{V} (\mathcal{L}(b_1,3b_2),\Vert \cdot \Vert_1)$: Points that are strictly closer to $0$ than to any other lattice vector are light gray, and points which have the same distance to some other lattice vector are dark gray.}
  \label{fig:Voronoi_cell_1norm}
\end{figure}
\begin{proofofGVRdefineVC}
  It is clear that $\mathcal{V}(\Lambda, \Vert \cdot \Vert) \subseteq \tilde{\mathcal{V}}^{(g)}(\Lambda, \Vert \cdot \Vert)$.
  For $x \in \mathbb{R}^n \setminus \mathcal{V}(\Lambda, \Vert \cdot \Vert)$, we show in the following that $x \notin \tilde{\mathcal{V}}^{(g)}(\Lambda, \Vert \cdot \Vert)$.
  Since $\Lambda$ is discrete, we have $\lbrace u \in \Lambda \mid \Vert x-u \Vert < \Vert x \Vert \rbrace = \lbrace u_1, \ldots, u_k \rbrace$ for some $k \in \mathbb{N}$. Note that this set is not empty due to $x \notin \mathcal{V}(\Lambda, \Vert \cdot \Vert)$.
  By the continuity of $y \mapsto \Vert y \Vert$ and the intermediate value theorem, we find for every $1 \leq i \leq k$ some $\tau_i \in (0,1)$ with $\Vert \tau_i x - u_i \Vert = \tau_i \Vert x \Vert$.
  Now we choose $1 \leq j \leq k$ such that $\tau_j = \min \lbrace \tau_1, \ldots, \tau_k \rbrace$.
  The first part of Lemma~\ref{lem:star} implies $\tau_j x \in \mathcal{V}(\Lambda, \Vert \cdot \Vert)$.
  This shows that $u_j$ is a weak Voronoi-relevant vector, and from $\Vert x - u_j \Vert < \Vert x \Vert$ we get $x \notin \tilde{\mathcal{V}}^{(g)}(\Lambda, \Vert \cdot \Vert)$.
\end{proofofGVRdefineVC}
Theorem~\ref{thm:bijection_facets_Voronoi-relevant} is not true for non-strictly convex norms, not even in the two-dimensional case. To see this, consider the same lattice $\mathcal{L}(b_1,3b_2)$ together with the $\ell_1$-norm as in Figure~\ref{fig:Voronoi_cell_1norm}. Two facets of $\mathcal{V}(\mathcal{L}(b_1,3b_2),\Vert \cdot \Vert_1)$ are shown in Figure~\ref{fig:facets}, but only the facet in Figure~\ref{sfig:facet1} is of the form as in Theorem~\ref{thm:bijection_facets_Voronoi-relevant}.
Figures~\ref{sfig:facet3} and~\ref{sfig:facet4} show the reason for the third condition in Definition~\ref{defn:facet}.
\begin{figure}[tbhp]
  \centering
  \subfloat[]{\label{sfig:facet1}\includegraphics[width=0.24\textwidth]{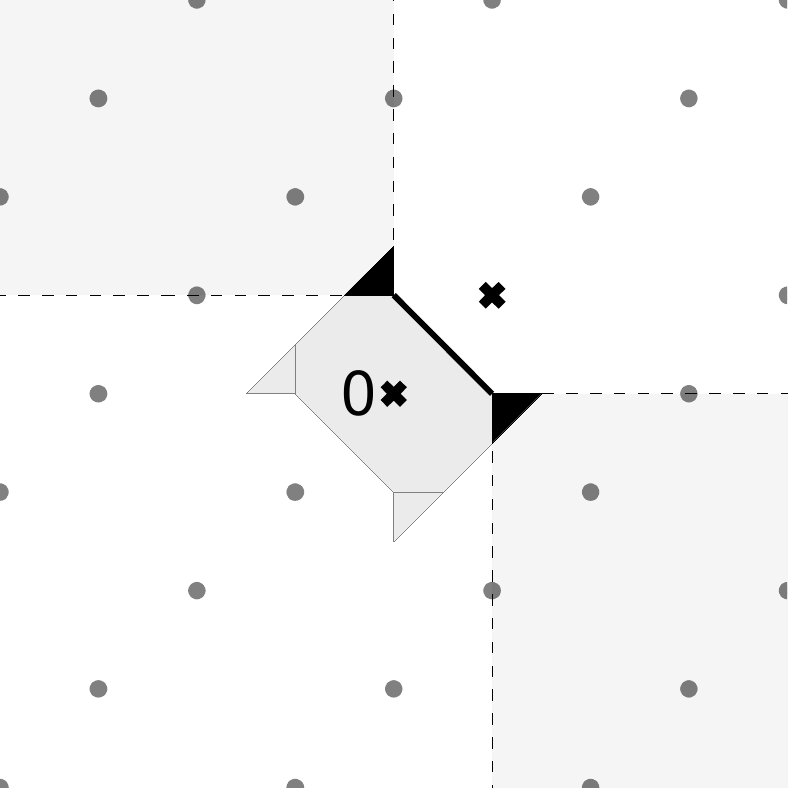}}~
  \subfloat[]{\label{sfig:facet2}\includegraphics[width=0.24\textwidth]{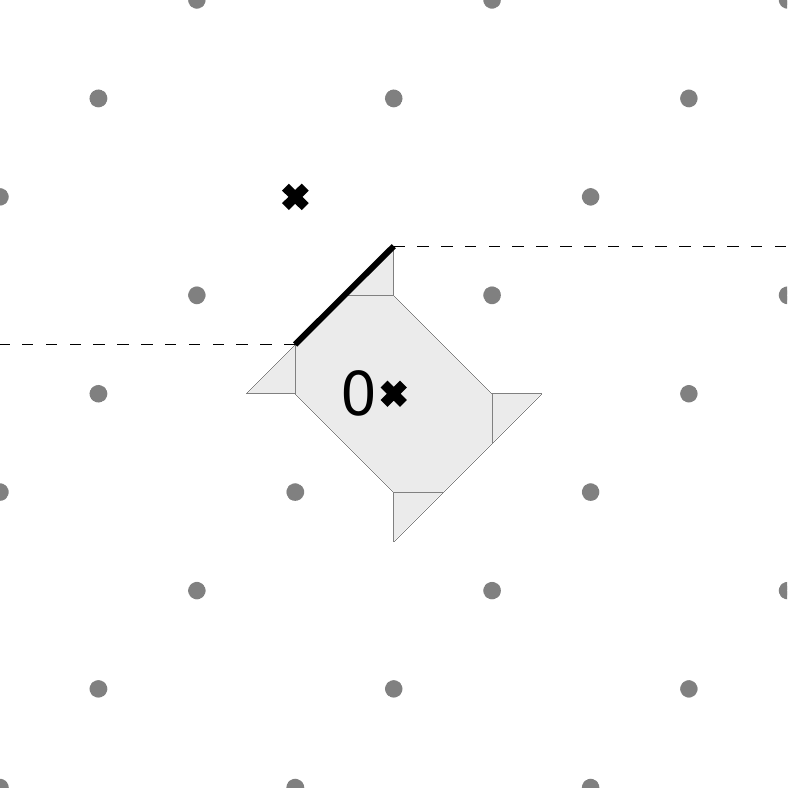}}~
  \subfloat[]{\label{sfig:facet3}\includegraphics[width=0.24\textwidth]{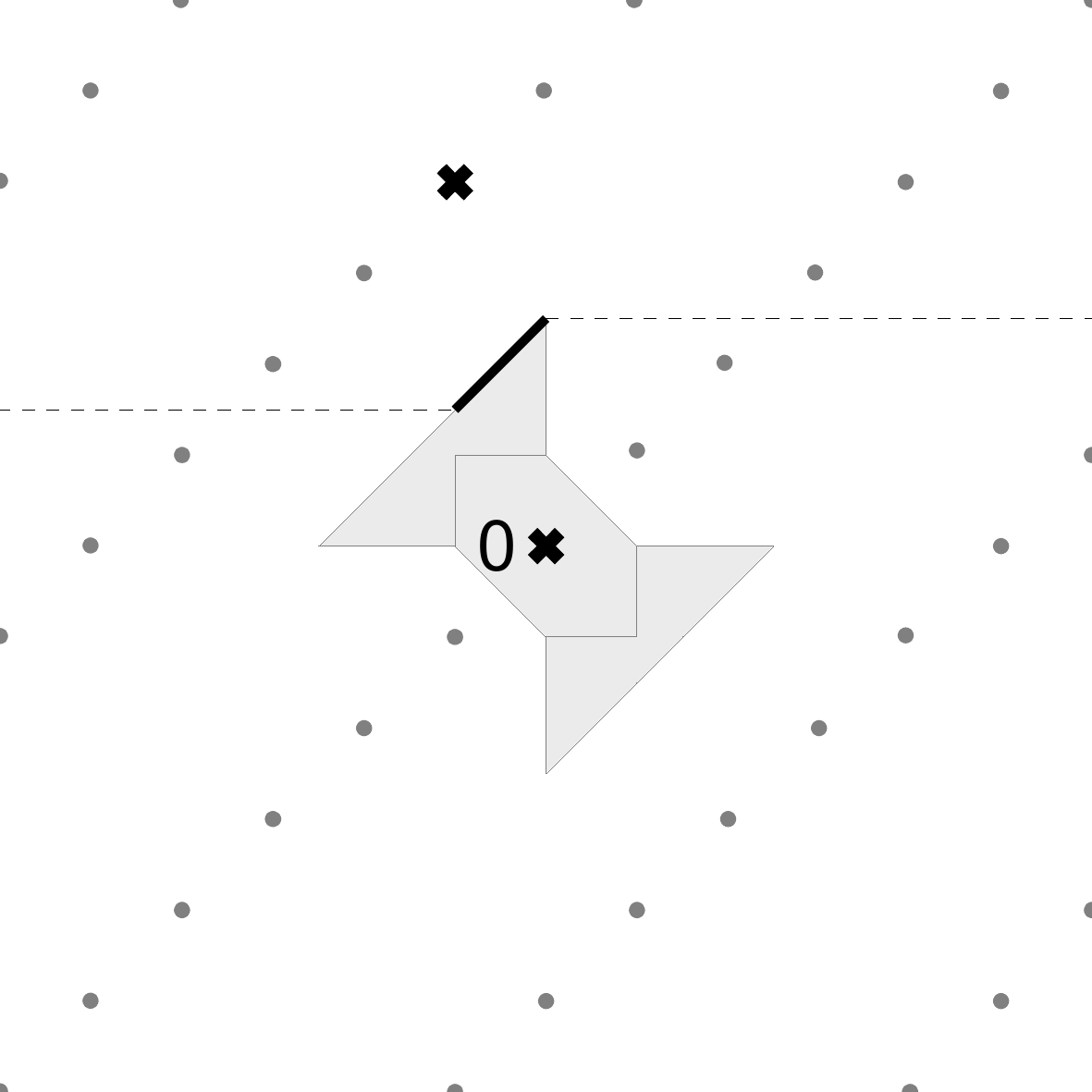}}~
  \subfloat[]{\label{sfig:facet4}\includegraphics[width=0.24\textwidth]{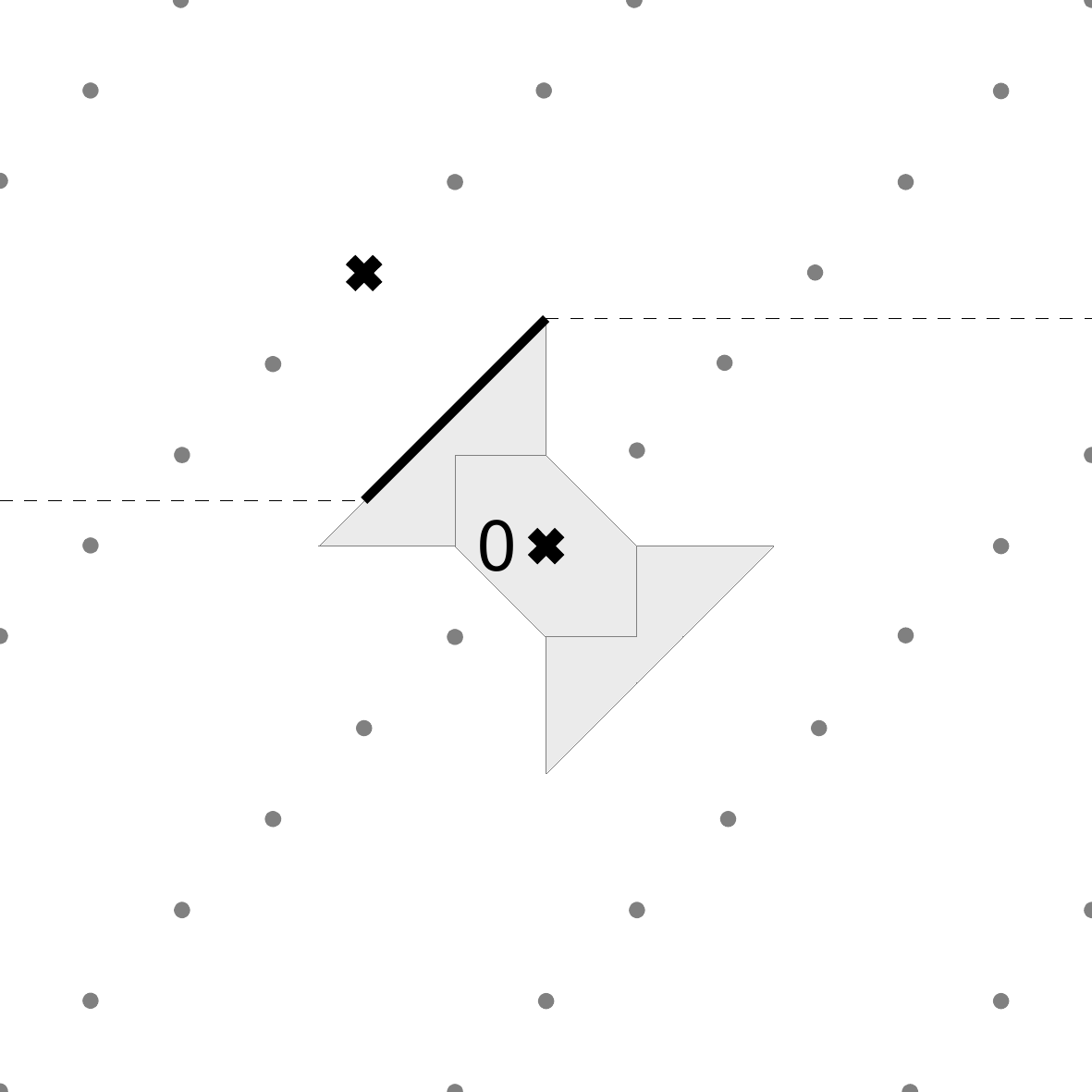}}
  \caption{Facets of $\mathcal{V} (\mathcal{L}(b_1,m b_2),\Vert \cdot \Vert_1)$ (for $m \in \lbrace 3,5 \rbrace$) and lattice vectors inducing them are black. Dashed lines indicate bisectors which contain facets. Figure~\ref{sfig:facet3} actually shows a non-facet contained in the facet in Figure~\ref{sfig:facet4}.}
\label{fig:facets}
\end{figure}

At least we can show Proposition~\ref{prop:Voronoi_relevant_has_facet}, which states for all norms that every Voronoi-relevant vector induces a facet of the Voronoi cell which is of the form as in Theorem~\ref{thm:bijection_facets_Voronoi-relevant}.
This result will also be important for the proof of Theorem~\ref{thm:bijection_facets_Voronoi-relevant} itself.

\begin{proofofVRinduceFacets}
  Let $v \in \Lambda$ be Voronoi-relevant.
  First, we verify that $\mathcal{F}:=\mathcal{V}(\Lambda, \Vert \cdot \Vert) \cap \mathcal{H}_{\Vert \cdot \Vert}^{=}(0,v)$ satisfies the third condition of Definition~\ref{defn:facet}.
  Since $v$ is Voronoi-relevant, there is some $x \in \mathbb{R}^n$ such that $\Vert x \Vert = \Vert x-v \Vert < \Vert x-w \Vert$ for every $w \in \Lambda \setminus \lbrace 0,v \rbrace$. Hence, $x \in \mathcal{F}$ follows, and for all $w \in \Lambda \setminus \lbrace 0,v \rbrace$ we have $x \notin \mathcal{H}_{\Vert \cdot \Vert}^{=}(0,w)$ and thus $\mathcal{F} \nsubseteq \mathcal{H}_{\Vert \cdot \Vert}^{=}(0,w)$. 

Secondly, we show that $\mathcal{F}$ satisfies the second condition of Definition~\ref{defn:facet}.
  Because $\Lambda$ is discrete, there is some $\varepsilon \in \mathbb{R}_{>0}$ such that $\mathcal{B}_{\Vert \cdot \Vert, \Vert x \Vert + \varepsilon}(x) \cap \Lambda = \lbrace 0,v \rbrace$. The continuity of $\Vert \cdot \Vert$ with respect to the Euclidean norm yields $\delta_1, \delta_2 \in \mathbb{R}_{>0}$ such that $|\Vert x \Vert-\Vert y_1 \Vert| < \frac{\varepsilon}{2}$ holds for every $y_1 \in \mathcal{B}_{\Vert \cdot \Vert_2, \delta_1}(x)$ and $\Vert y_2-x \Vert < \frac{\varepsilon}{2}$ holds for every $y_2 \in \mathcal{B}_{\Vert \cdot \Vert_2, \delta_2}(x)$. Define $\delta := \min \lbrace \delta_1, \delta_2 \rbrace$ and let $y \in \mathcal{B}_{\Vert \cdot \Vert_2, \delta}(x) \cap \mathcal{H}_{\Vert \cdot \Vert}^{=}(0,v)$. It follows that $\overline{\mathcal{B}}_{\Vert \cdot \Vert, \Vert y \Vert}(y) \subseteq \mathcal{B}_{\Vert \cdot \Vert, \Vert x \Vert + \varepsilon}(x)$, which leads to $\overline{\mathcal{B}}_{\Vert \cdot \Vert, \Vert y \Vert}(y) \cap \Lambda = \lbrace 0,v \rbrace$. Thus, we have $\Vert y \Vert = \Vert y-v \Vert < \Vert y-w \Vert$ for all $w \in \Lambda \setminus \lbrace 0,v \rbrace$, and $y \in \mathcal{F}$.
\end{proofofVRinduceFacets}

\subsection{Strictly Convex Norms}
\label{ssec:strNorms}
In the following, we only consider strictly convex norms.
We show that the Voronoi-relevant vectors define the Voronoi cell and that the Voronoi-relevant vectors are in bijection with the facets of the Voronoi cell, if the lattice has dimension two or the norm is smooth.
For this, we need to understand bisectors and their intersections. These were for example studied in \cite{bisector_hyperplane, dissertation}.
\begin{proposition}[\cite{bisector_hyperplane}, Theorem 2, and \cite{dissertation}, Theorem 2.1.2.3]
\label{prop:trisector_2dimensions}
$ $
\begin{enumerate}
\item Let $\Vert \cdot \Vert: \mathbb{R}^n \to \mathbb{R}_{\geq 0}$ be a strictly convex norm. For distinct $a,b \in \mathbb{R}^n$, the bisector $\mathcal{H}_{\Vert \cdot \Vert}^{=}(a,b)$ is homeomorphic to a hyperplane.
\item Let $\Vert \cdot \Vert: \mathbb{R}^2 \to \mathbb{R}_{\geq 0}$ be any norm, and let $a,b,c \in \mathbb{R}^2$ be pairwise distinct such that each of the bisectors $\mathcal{H}_{\Vert \cdot \Vert}^{=}(a,b)$, $\mathcal{H}_{\Vert \cdot \Vert}^{=}(a,c)$ and $\mathcal{H}_{\Vert \cdot \Vert}^{=}(b,c)$ is homeomorphic to a line. Then $\mathcal{H}_{\Vert \cdot \Vert}^{=}(a,b) \cap \mathcal{H}_{\Vert \cdot \Vert}^{=}(b,c)$ is either empty or a single point.
\end{enumerate} 
\end{proposition}
We will prove in Proposition~\ref{prop:trisector_is_(n-2)-dimensional} that the intersection of bisectors of three non-collinear points in $\mathbb{R}^n$ is an $(n-2)$-dimensional manifold. This result is interesting on its own, especially when considering Proposition~\ref{prop:trisector_2dimensions}, but it will also help to prove that the Voronoi-relevant vectors determine Voronoi cells for strictly convex and smooth norms.
To get these results, we first need that smooth norms are continuously differentiable functions. For this, we use a result from \cite{schneider} and express our statement in Corollary~\ref{cor:norm_function_is_C1} using convex duality from Proposition~\ref{prop:duality}.
\begin{proposition}[\cite{schneider}, Corollary 1.7.3]
\label{Schneider:support_function_differentiable}
Let $K \subseteq \mathbb{R}^n$ be a convex body and $u \in \mathbb{R}^n \setminus \lbrace 0 \rbrace$. The support function $h_K$ is differentiable at $u$ if and only if there exists exactly one $x \in K$ with $h_K(u) = \langle x,u \rangle$. In this case, $\nabla h_K(u) = x$.
\end{proposition}
\begin{corollary}
\label{cor:norm_function_is_C1}
For every strictly convex body $K \subseteq \mathbb{R}^n$, we have that $h_K$ is continuously differentiable on $\mathbb{R}^n \setminus \lbrace 0 \rbrace.$
\end{corollary}
\begin{proof}
  Let $u \in \mathbb{R}^n \setminus \lbrace 0 \rbrace$. Since $K$ is compact and $x \mapsto \langle x,u \rangle$ is continuous, there exists some $x \in K$ with $h_K(u) = \langle x,u \rangle$. In particular, $x \in \partial K$ and $\lbrace y \in \mathbb{R}^n \mid \langle y,u \rangle = h_K(u) \rbrace$ is a supporting hyperplane of $K$ at $x$. Since $K$ is strictly convex, supporting hyperplanes at distinct points of $\partial K$ are distinct, which implies that $x$ is the unique point in $K$ with $h_K(u) = \langle x,u \rangle$. By Proposition~\ref{Schneider:support_function_differentiable}, $h_K$ is differentiable at $u$.

  It is left to show that $\nabla h_K: \mathbb{R}^n \setminus \lbrace 0 
  \rbrace \to \mathbb{R}^n$ is continuous with respect to the Euclidean norm. Consider $u \in \mathbb{R}^n \setminus \lbrace 0 \rbrace$ and $\varepsilon > 0$. Assume for contradiction that for every $\delta >0$ one finds $v \in \mathbb{R}^n \setminus \lbrace 0 \rbrace$ with $\Vert v-u \Vert_2 < \delta$ but $\Vert \nabla h_K(v) - \nabla h_K(u) \Vert_2 \geq \varepsilon$. By Proposition~\ref{Schneider:support_function_differentiable}, $\nabla h_K(v) \in \partial K$ for every $v \in \mathbb{R}^n \setminus \lbrace 0 \rbrace$. Since $\partial K$ is compact, there is a sequence $(v_i)_{i \in \mathbb{Z}_{>0}} \subseteq \mathbb{R}^n \setminus \lbrace 0 \rbrace$ such that $\lim_{i \to \infty} v_i = u$, $z := \lim_{i \to \infty} \nabla h_K(v_i) \in \partial K$ and $\Vert \nabla h_K(v_i) - \nabla h_K(u) \Vert_2 \geq \varepsilon$ holds for every $i \in \mathbb{Z}_{>0}$. On the one hand, this implies $\Vert z - \nabla h_K(u) \Vert_2 \geq \varepsilon$.
  On the other hand, we get from Proposition~\ref{Schneider:support_function_differentiable} that $ \lim_{i \to \infty} \langle v_i, \nabla h_K(v_i) \rangle = \lim_{i \to \infty} h_K(v_i) = h_K(u)$
  and from the Cauchy--Schwarz inequality that $|\langle v_i, z- \nabla h_K(v_i) \rangle| \leq \Vert v_i \Vert_2 \cdot \Vert z- \nabla h_K(v_i) \Vert_2$, which yield
  \begin{align*}
    \langle u, z- \nabla h_K(u) \rangle
    &= \lim_{i \to \infty} \langle v_i, z- \nabla h_K(v_i) \rangle
    + \lim_{i \to \infty} \langle v_i, \nabla h_K(v_i) \rangle
    - \langle u, \nabla h_K(u) \rangle \\
    &= 0 + h_K(u) - h_K(u) = 0.
  \end{align*}
  Again by Prooposition~\ref{Schneider:support_function_differentiable}, we have $\langle u,z \rangle = h_K(u)$ and $z = \nabla h_K(u)$, which contradicts $\Vert z - \nabla h_K(u) \Vert_2 \geq \varepsilon$.
\end{proof}
To prove that the intersection of bisectors as described above is a manifold of proper dimension, we need the following simple version of the Regular Level Set Theorem.
\begin{proposition}
\label{prop:regular_level_set_theorem}
Let $O \subseteq \mathbb{R}^n$ be open and $F: O \to \mathbb{R}^d$ be continuously differentiable with $0 < d < n$. Define $M := \lbrace x \in O \mid F(x)=0 \rbrace$, and denote by $J_F$ the Jacobian $(d \times n)$-matrix of all first-order partial derivatives of $F$, i.e., the $(i,j)$-th entry of $J_F$ is $\frac{\partial F_i}{\partial x_j}$.
If, for every $x \in M$, the  matrix $J_F(x)$ has full rank, $M$ is an $(n-d)$-dimensional manifold.
\end{proposition}
\begin{proof}
  Let $x \in M$. Since $J_F(x)$ has full rank, we can assume without loss of generality that the last $d$ columns of $J_F(x)$ form an invertible matrix. By the implicit function theorem, there exist open neighborhoods $U$ and $V$ of $(x_1, \ldots, x_{n-d})$ and $(x_{n-d+1}, \ldots, x_n)$, respectively, with $U \times V \subseteq O$ as well as a unique continuously differentiable function $g:U \to V$ such that
  $\lbrace (y,g(y)) \mid y \in U \rbrace = \lbrace (y,z) \in U \times V \mid F(y,z)=0 \rbrace$. Hence, $\varphi: U \to M \cap (U \times V), y \mapsto (y, g(y))$ is a homeomorphism.
\end{proof}
\begin{proposition}
\label{prop:trisector_is_(n-2)-dimensional}
Let $\Vert \cdot \Vert: \mathbb{R}^n \to \mathbb{R}_{\geq 0}$ be a strictly convex and smooth norm with $n>2$. For all non-collinear $a,b,c \in \mathbb{R}^n$, we have that $\mathcal{H}_{\Vert \cdot \Vert}^{=}(a,b) \cap \mathcal{H}_{\Vert \cdot \Vert}^{=}(b,c)$ is an $(n-2)$-dimensional manifold.
\end{proposition}
\begin{proof}
  Let $K$ be the closed unit ball of $\Vert \cdot \Vert$. Then~-- by Proposition~\ref{prop:duality}~-- the dual body $K^\circ$ is also strictly convex and smooth, and $\Vert \cdot \Vert = h_{K^\circ}$.
  By Corollary~\ref{cor:norm_function_is_C1},
  \begin{align*}
    F: \mathbb{R}^n \setminus \lbrace a,b,c \rbrace &\longrightarrow \mathbb{R}^2, \\
    x &\longmapsto \left( \begin{array}{c} 
      \Vert x-a \Vert - \Vert x-b \Vert \\
      \Vert x-b \Vert - \Vert x-c \Vert
    \end{array} \right)
  \end{align*}
  is continuously differentiable. By Proposition~\ref{prop:regular_level_set_theorem}, it is enough to show that $J_F(x)$ has rank two for every $x \in \mathcal{H}_{\Vert \cdot \Vert}^{=}(a,b) \cap \mathcal{H}_{\Vert \cdot \Vert}^{=}(b,c)$. Hence, consider such an $x$ in the following. For $p \in \lbrace a,b,c \rbrace$ we set $y_p := \nabla h_{K^\circ} (x-p)$.

  First, we assume for contradiction that one of the three equalities $y_a = y_b$, $y_a = y_c$ or $y_b = y_c$ holds. Without loss of generality, let $y_a = y_b =: y$. By Proposition~\ref{Schneider:support_function_differentiable}, $y$ is the unique point in $K^\circ$ with $\langle y,x-a \rangle = \Vert x-a \Vert = \Vert x-b \Vert = \langle y, x-b \rangle$. Thus, $\lbrace z \in \mathbb{R}^n \mid \langle z,x-a \rangle = \Vert x-a \Vert \rbrace$ and $\lbrace z \in \mathbb{R}^n \mid \langle z,x-b \rangle = \Vert x-b \Vert \rbrace$ are supporting hyperplanes of $K^\circ$ at $y$. The smoothness of $K^\circ$ implies the equality of both hyperplanes. From this it follows that $x-a$ and $x-b$ need to be linearly dependent, and due to $\Vert x-a \Vert = \Vert x-b  \Vert$ we have $a = b$, which, as shown in the previous paragraph, contradicts the non-collinearity of $a,b,c$.

  Secondly, we assume that the two rows of $J_F(x)$ are linearly dependent, i.e., there exists some $\lambda \in \mathbb{R}$ such that $y_b - y_c = \lambda (y_a-y_b)$. This yields $y_c = -\lambda y_a + (1+\lambda) y_b$. Since $K^\circ$ is strictly convex and $y_a$, $y_b$ and $y_c$ lie on the boundary of $K^\circ$, we have $-\lambda \in \lbrace 0,1 \rbrace$. This implies $y_c \in \lbrace y_a, y_b \rbrace$, which, as shown in the previous paragraph, contradicts the non-collinearity of $a,b,c$. Hence, $J_F(x)$ has rank two.
\end{proof}
We need one more ingredient to show that the Voronoi-relevant vectors determine the Voronoi cell, namely that the boundary of a Voronoi cell (which only consists of bisector parts) is $(n-1)$-dimensional.
\begin{proposition}
\label{prop:Voronoi_cell_boundary_homeomorphic_to_sphere}
For every lattice $\Lambda \subseteq \mathbb{R}^n$ and every strictly convex norm $\Vert \cdot \Vert$, the boundary of the Voronoi cell $\mathcal{V}(\Lambda, \Vert \cdot \Vert)$ is homeomorphic to the $(n-1)$-dimensional sphere $S^{n-1}$.
\end{proposition}
\begin{proof}
  The Voronoi cell $\mathcal{V}(\Lambda, \Vert \cdot \Vert)$ is clearly bounded. Since halfspaces of the form $\mathcal{H}_{\Vert \cdot \Vert}^{<}(v,0)$ are open, $\mathcal{V}(\Lambda, \Vert \cdot \Vert)$ is also closed. Thus, the Voronoi cell and its boundary are compact. Furthermore, the boundary of the Voronoi cell is given by
  \begin{align}
    \label{eq:boundary_Voronoi_cell}
    \partial\mathcal{V}(\Lambda, \Vert \cdot \Vert) = \left\lbrace x \in \mathcal{V}(\Lambda, \Vert \cdot \Vert) \mid \exists v \in \Lambda \setminus \lbrace 0 \rbrace: \Vert x \Vert = \Vert x - v \Vert \right\rbrace.
  \end{align}
  Indeed, given an $x \in \mathcal{V}(\Lambda, \Vert \cdot \Vert)$, which is not contained in the right hand side of~\eqref{eq:boundary_Voronoi_cell}, we have for all weak Voronoi-relevant vectors $v \in \Lambda$ that $\Vert x \Vert < \Vert x-v \Vert$. Because all $\mathcal{H}_{\Vert \cdot \Vert}^{<}(0,v)$ are open, we find $\varepsilon_v \in \mathbb{R}_{>0}$ with $\mathcal{B}_{\Vert \cdot \Vert_2, \varepsilon_v}(x) \subseteq \mathcal{H}_{\Vert \cdot \Vert}^{<}(0,v)$. Since there are only finitely many weak Voronoi-relevant vectors (Proposition~\ref{prop:trivial_upper_bound}) and these define the Voronoi cell (Theorem~\ref{thm:generalized_Voronoi-relevant_vectors_define_Voronoi-cell}), we can choose the minimal $\varepsilon$ of all these $\varepsilon_v$. For this $\varepsilon$  we obtain $\mathcal{B}_{\Vert \cdot \Vert_2, \varepsilon}(x) \subseteq \mathcal{V}(\Lambda, \Vert \cdot \Vert)$. Hence, $x \notin \partial \mathcal{V}(\Lambda, \Vert \cdot \Vert)$. The other inclusion (``$\supseteq$'') follows from Lemma~\ref{lem:star} by considering the ray $\lbrace \tau x \mid \tau \geq 0 \rbrace$ from $0$ through $x$.

  The desired homeomorphism $\varphi: \partial \mathcal{V}(\Lambda, \Vert \cdot \Vert) \to S^{n-1}$ is now given by the central projection of the boundary of the Voronoi cell on the sphere. This projection maps every point $x \in \partial \mathcal{V}(\Lambda, \Vert \cdot \Vert)$ to the unique intersection point $x'$ of $S^{n-1}$ and the ray from 0 though $x$.

  For distinct $x,y \in \partial \mathcal{V}(\Lambda, \Vert \cdot \Vert)$ with $\varphi(x)=\varphi(y)$, $x$ and $y$ have to be linearly dependent. Without loss of generality, we can assume that $x = \lambda y$ for some $\lambda \in (0,1)$.
  From Lemma~\ref{lem:star} and \eqref{eq:boundary_Voronoi_cell} we get that $x \notin \partial \mathcal{V}(\Lambda, \Vert \cdot \Vert)$, which contradicts our assumption. Therefore, $\varphi$ is injective.
  For any $x' \in S^{n-1}$, one can define $\tau := \sup \lbrace \lambda \in \mathbb{R}_{>0} \mid \lambda x' \in \intt \left( \mathcal{V}(\Lambda, \Vert \cdot \Vert) \right)\rbrace$ such that $\tau x' \in \partial \mathcal{V}(\Lambda, \Vert \cdot \Vert)$ with $\varphi (\tau x') = x'$. Hence, $\varphi$ is bijective. It is for example shown in \cite{kelly_weiss} that $\varphi$ is continuous. Since $\varphi$ is a continuous bijection from a compact space onto a Hausdorff space, it is already a homeomorphism (e.g., Corollary 2.4 in Chapter 7 of~\cite{homeomorphism}).
\end{proof}
Now we can show that the Voronoi-relevant vectors of a lattice $\Lambda \subseteq \mathbb{R}^n$ define its Voronoi cell. We show this first for the \emph{strict Voronoi cell}
\begin{align*}
  \mathcal{V}^{(i)}(\Lambda, \Vert \cdot \Vert) := \lbrace x \in \mathbb{R}^n \mid \forall v \in \Lambda \setminus \lbrace 0 \rbrace: \Vert x \Vert < \Vert x-v \Vert \rbrace.
\end{align*}
\begin{theorem}
\label{thm:Voronoi-relevant_vectors_define_strict_Voronoi-cell}
For every lattice $\Lambda \subseteq \mathbb{R}^n$ and every strictly convex and smooth norm $\Vert \cdot \Vert$, the strict Voronoi cell $\mathcal{V}^{(i)}(\Lambda, \Vert \cdot \Vert)$ is equal to
\begin{align*}
  \tilde{\mathcal{V}}^{(i)}(\Lambda, \Vert \cdot \Vert)
  := \left\lbrace x \in \mathbb{R}^n \;\middle\vert\; \begin{array}{l}
    \forall v \in \Lambda \text{ Voronoi-relevant with} \\
    \text{respect to } \Vert \cdot \Vert:  \Vert x \Vert < \Vert x-v \Vert
  \end{array} \right\rbrace.
\end{align*}
For two-dimensional lattices, smoothness of the norm is not necessary.
\end{theorem}
\begin{proof}
  We first assume that the underlying norm is strictly convex and smooth. It is clear that $\mathcal{V}^{(i)}(\Lambda, \Vert \cdot \Vert) \subseteq \tilde{\mathcal{V}}^{(i)}(\Lambda, \Vert \cdot \Vert)$. For the other direction, let $x \in \tilde{\mathcal{V}}^{(i)}(\Lambda, \Vert \cdot \Vert)$ and assume for contradiction that $x \notin \mathcal{V}^{(i)}(\Lambda, \Vert \cdot \Vert)$, i.e., there exists some $u \in \Lambda \setminus \lbrace 0 \rbrace$ with $\Vert x-u \Vert \leq \Vert x \Vert$.

  If $x \notin \mathcal{V}(\Lambda, \Vert \cdot \Vert)$, let $k \in \mathbb{Z}_{>0}$ with $\lbrace u \in \Lambda \mid \Vert x-u \Vert < \Vert x \Vert \rbrace = \lbrace u_1, \ldots, u_k \rbrace$.
  As in the proof of Theorem~\ref{thm:generalized_Voronoi-relevant_vectors_define_Voronoi-cell}, we find for every $1 \leq i \leq k $ some $\tau_i \in (0,1)$ with $\Vert \tau_i x - u_i \Vert = \tau_i \Vert x \Vert$, and we pick $1 \leq j \leq k$ such that $\tau_j = \min \lbrace \tau_1, \ldots, \tau_k \rbrace$ and $y := \tau_j x \in \mathcal{V}(\Lambda, \Vert \cdot \Vert) \setminus \mathcal{V}^{(i)}(\Lambda, \Vert \cdot \Vert)$.
  If $x \in \mathcal{V}(\Lambda, \Vert \cdot \Vert)$, set directly $y := x$.

  Now we write $\lbrace w \in \Lambda \setminus \lbrace 0 \rbrace \mid \Vert y-w \Vert =  \Vert y \Vert \rbrace = \lbrace w_1, \ldots, w_l \rbrace$ for $l \in \mathbb{Z}_{>0}$. We have $\overline{\mathcal{B}}_{\Vert \cdot \Vert, \Vert y \Vert}(y) \cap \Lambda = \lbrace 0, w_1, \ldots, w_l \rbrace$
  and can use the same ideas as in the proof of Proposition~\ref{prop:Voronoi_relevant_has_facet}.
  Since $\Lambda$ is discrete, there is an $\varepsilon \in \mathbb{R}_{>0}$ such that $\mathcal{B}_{\Vert \cdot \Vert, \Vert y \Vert+\varepsilon}(y) \cap \Lambda = \lbrace 0, w_1, \ldots, w_l \rbrace$. By the continuity of $\Vert \cdot \Vert$ with respect to the Euclidean norm, we find $\delta_1, \delta_2 \in \mathbb{R}_{>0}$ such that $|\Vert y \Vert - \Vert z_1 \Vert| < \frac{\varepsilon}{2}$ holds for every $z_1 \in \mathcal{B}_{\Vert \cdot \Vert_2, \delta_1}(y)$ and $\Vert z_2-y \Vert < \frac{\epsilon}{2}$ holds for every $z_2 \in \mathcal{B}_{\Vert \cdot \Vert_2, \delta_2}(y)$. Define $\delta := \min \lbrace \delta_1, \delta_2 \rbrace$. We have for every $z \in \mathcal{B}_{\Vert \cdot \Vert_2, \delta}(y)$ that
  \begin{align}
    \label{eq:small_environment_around_y}
    \overline{\mathcal{B}}_{\Vert \cdot \Vert, \Vert z \Vert}(z) \subseteq \mathcal{B}_{\Vert \cdot \Vert, \Vert y \Vert+\varepsilon}(y).
  \end{align}
  Locally around $y$, the set
  \begin{align}
    \label{eq:definition_of_S}
    \mathcal{S} := \left\lbrace z \in \mathbb{R}^n \mid
     \forall \, 1 \leq i \leq l: \Vert z \Vert \leq \Vert z-w_i \Vert,
      \exists \, 1 \leq m \leq l: \Vert z \Vert = \Vert z-w_m \Vert
       \right\rbrace
     \end{align}
     coincides with the boundary of the Voronoi cell, i.e., $\mathcal{S} \cap \mathcal{B}_{\Vert \cdot \Vert_2, \delta}(y) = \partial \mathcal{V}(\Lambda, \Vert \cdot \Vert) \cap \mathcal{B}_{\Vert \cdot \Vert_2, \delta}(y)$.
     By Proposition~\ref{prop:Voronoi_cell_boundary_homeomorphic_to_sphere}, $\mathcal{S} \cap \mathcal{B}_{\Vert \cdot \Vert_2, \delta}(y)$ is an $(n-1)$-dimensional manifold. Together with Proposition~\ref{prop:trisector_is_(n-2)-dimensional}, we find some
     \begin{align}
       \label{eq:corners_of_Voronoi_cell_have_points_on_exactly_one_bisector_in_neighborhood}
       z \in \left( \mathcal{S} \cap \mathcal{B}_{\Vert \cdot \Vert_2, \delta}(y) \right) \setminus 
       \left( \bigcup \limits_{1 \leq i_1 < i_2 \leq l} \left( \mathcal{H}_{\Vert \cdot \Vert}^{=} (0, w_{i_1}) \cap \mathcal{H}_{\Vert \cdot \Vert}^{=} (0, w_{i_2}) \right) \right).
     \end{align}
     Thus, there is some $i \in \lbrace 1, \ldots, l \rbrace$ with $\Vert z \Vert = \Vert z-w_i \Vert$. By \eqref{eq:small_environment_around_y}, \eqref{eq:definition_of_S} and \eqref{eq:corners_of_Voronoi_cell_have_points_on_exactly_one_bisector_in_neighborhood}, we have for every $v \in \Lambda \setminus \lbrace 0, w_i \rbrace$ that $\Vert z-v \Vert > \Vert z \Vert$. This means that $w_i$ is Voronoi-relevant, which contradicts $y \in \tilde{\mathcal{V}}^{(i)}(\Lambda, \Vert \cdot \Vert)$.
  This concludes the proof for strictly convex and smooth norms.

     We only used the smoothness assumption to show~\eqref{eq:corners_of_Voronoi_cell_have_points_on_exactly_one_bisector_in_neighborhood} with the help of Proposition~\ref{prop:trisector_is_(n-2)-dimensional}. For two-dimensional lattices, we obtain \eqref{eq:corners_of_Voronoi_cell_have_points_on_exactly_one_bisector_in_neighborhood} from Proposition~\ref{prop:trisector_2dimensions} without smoothness.
   \end{proof}
\begin{proofofVRdefineVC}
  Let $x \in \tilde{\mathcal{V}}(\Lambda, \Vert \cdot \Vert)$ and assume for contradiction that $x \notin \mathcal{V}(\Lambda, \Vert \cdot \Vert)$, i.e., there is some $w \in \Lambda \setminus \lbrace 0 \rbrace$ with $\Vert x \Vert > \Vert x-w \Vert$.
  By the continuity of $y \mapsto \Vert y \Vert$ and the intermediate value theorem, there is $\tau \in (0,1)$ with $\Vert \tau x -w \Vert = \tau \Vert x \Vert$. Lemma~\ref{lem:star} shows that $\tau x \in \tilde{\mathcal{V}}^{(i)}(\Lambda, \Vert \cdot \Vert)$.
  By Theorem~\ref{thm:Voronoi-relevant_vectors_define_strict_Voronoi-cell}, we have $\tau x \in {\mathcal{V}}^{(i)}(\Lambda, \Vert \cdot \Vert)$ and in particular $\tau \Vert x \Vert < \Vert \tau x -w \Vert$, which contradicts the choice of $\tau$.
\end{proofofVRdefineVC}
Finally, we show the bijection between Voronoi-relevant vectors and facets of the Voronoi cell. 
Note that the third condition of Definition~\ref{defn:facet} is not needed under the assumptions of Theorem~\ref{thm:bijection_facets_Voronoi-relevant}. This follows from Propositions~\ref{prop:trisector_2dimensions} and~\ref{prop:trisector_is_(n-2)-dimensional}, as can be seen in the following proof.
\begin{proofofbijection}
  By Proposition~\ref{prop:Voronoi_relevant_has_facet}, it is enough to show that, for every facet $\mathcal{F}$ of $\mathcal{V}(\Lambda, \Vert \cdot \Vert)$, there is a unique Voronoi-relevant vector $v \in \Lambda$ such that $\mathcal{F} = \mathcal{V}(\Lambda, \Vert \cdot \Vert) \cap \mathcal{H}_{\Vert \cdot \Vert}^{=}(0,v)$.
  We show this first for strictly convex and smooth norms.
  For a facet $\mathcal{F}$, there is some $v \in \Lambda \setminus \lbrace 0 \rbrace$ with $\mathcal{F} = \mathcal{V}(\Lambda, \Vert \cdot \Vert) \cap \mathcal{H}_{\Vert \cdot \Vert}^{=}(0,v)$.
  Moreover, there are $x \in \mathcal{F}$ and $\delta \in \mathbb{R}_{>0}$ with $\mathcal{B}_{\Vert \cdot \Vert_2, \delta}(x) \cap \mathcal{H}_{\Vert \cdot \Vert}^{=}(0,v) \subseteq \mathcal{F}$.  
  By Propositions~\ref{prop:trisector_2dimensions} and \ref{prop:trisector_is_(n-2)-dimensional},    
  \begin{align}
    \label{eq:measure0}
    \bigcup \limits_{u \in \Lambda \setminus \lbrace 0,v \rbrace} \left( \mathcal{H}_{\Vert \cdot \Vert}^{=}(0,v) \cap \mathcal{H}_{\Vert \cdot \Vert}^{=}(0,u) \right) \cap \mathcal{B}_{\Vert \cdot \Vert_2, \delta}(x)
  \end{align}
  has measure zero in $\mathcal{B}_{\Vert \cdot \Vert_2, \delta}(x) \cap \mathcal{H}_{\Vert \cdot \Vert}^{=}(0,v)$.
  Hence, there is some $y \in \mathcal{B}_{\Vert \cdot \Vert_2, \delta}(x) \cap \mathcal{H}_{\Vert \cdot \Vert}^{=}(0,v) \subseteq \mathcal{F}$ that is not contained in~\eqref{eq:measure0}.
  This shows that $v$ is the unique vector in $\Lambda \setminus \lbrace 0 \rbrace$ with $\mathcal{F} \subseteq \mathcal{H}_{\Vert \cdot \Vert}^{=}(0,v)$, and that $v$ is Voronoi-relevant.
  This concludes the proof for strictly convex and smooth norms.

  We only used the smoothness assumption to see that \eqref{eq:measure0} has measure zero in $\mathcal{B}_{\Vert \cdot \Vert_2, \delta}(x) \cap \mathcal{H}_{\Vert \cdot \Vert}^{=}(0,v)$.
  For two-dimensional lattices, this follows from Proposition~\ref{prop:trisector_2dimensions} without smoothness.
\end{proofofbijection}

\subsection{Two-Dimensional Lattices}
\label{ssec:2Dim}
First, we focus on strictly convex norms and prove Prooposition~\ref{prop:2dimensions_strictly_convex_exact}.
Gr\"unbaum and Shephard study tilings of the Euclidean plane in~\cite{gruenbaum}.
Such a \emph{tiling} is a countable family of closed sets $\mathcal{T} = \lbrace T_1, T_2, \ldots \rbrace$ (called \emph{tiles}) such that $\bigcup_{i \in \mathbb{N}} T_i = \mathbb{R}^2$ and $\intt(T_i) \cap \intt(T_j) = \emptyset$ for $i \neq j$.
The Voronoi cells around all lattice points of a two-dimensional lattice do not necessarily form a tiling since their interiors might overlap (see Figure~\ref{fig:Voronoi_cell_1norm}).
\begin{lemma}
\label{lem:tiling}
Given a two-dimensional lattice $\Lambda$ and a strictly convex norm $\Vert \cdot \Vert$, the family of all Voronoi cells $\lbrace \mathcal{V}(\Lambda, \Vert \cdot \Vert) + v \mid v \in \Lambda \rbrace$ is a tiling.
\end{lemma}
\begin{proof}
The Voronoi cells around all lattice points are clearly closed and cover the whole plane. 
By~\eqref{eq:boundary_Voronoi_cell}, the interior of $\mathcal{V}(\Lambda, \Vert \cdot \Vert)$ is the strict Voronoi cell $\mathcal{V}^{(i)}(\Lambda, \Vert \cdot \Vert)$. 
Hence, the interiors of two Voronoi cells around two distinct lattice points do not meet.
\end{proof}
In particular, Gr\"unbaum and Shephard discuss several types of well-behaved tilings. 
They call a tiling $\mathcal{T}$ \emph{normal} if the following three conditions hold:
\begin{enumerate}
\item Every tile in $\mathcal{T}$ is homeomorphic to the closed disc $\overline{\mathcal{B}}_{\Vert \cdot \Vert_2, 1}(0)$. 
\item There are $r,R \in \mathbb{R}_{>0}$ such that every tile in $\mathcal{T}$ contains a disc of radius $r$ and is contained in a disc of radius $R$.
\item The intersection of every two tiles in $\mathcal{T}$ is connected.
\end{enumerate}
\begin{lemma}
The tiling in Lemma~\ref{lem:tiling} is normal.
\end{lemma}
\begin{proof}
The homeomorphism from Proposition~\ref{prop:Voronoi_cell_boundary_homeomorphic_to_sphere} can be extended from the whole Voronoi cell to the disc $\overline{\mathcal{B}}_{\Vert \cdot \Vert_2, 1}(0)$.
This proves the first condition.
The Voronoi cell $\mathcal{V}(\Lambda, \Vert \cdot \Vert)$ is clearly bounded. 
Moreover, since the origin is in the interior of $\mathcal{V}(\Lambda, \Vert \cdot \Vert)$, this Voronoi cell contains some disc with positive radius around the origin.
This shows condition two.

For the third condition, we assume for contradiction that the intersection of the two Voronoi cells $\mathcal{V}_v := \mathcal{V}(\Lambda, \Vert \cdot \Vert)+v$ and $\mathcal{V}_w := \mathcal{V}(\Lambda, \Vert \cdot \Vert)+w$ for distinct $v,w \in \Lambda$ is not connected.
This intersection is contained in the bisector $\mathcal{H}_{\Vert \cdot \Vert}^{=}(v,w)$, which is the image of a homeomorphism $\varphi$ with domain $\mathbb{R}$ by Proposition~\ref{prop:trisector_2dimensions}.
The preimage $\varphi^{-1}(\mathcal{V}_v \cap \mathcal{V}_w)$ consists of at least two disjoint closed intervals $I_1$ and $I_2$.
Without loss of generality, all points in $I_1$ are smaller than points in $I_2$.
Let $y_1$ be the maximal value in $I_1$ and $y_2$ be the minimal value in $I_2$. 
We pick a point $x \in (y_1, y_2)$ such that $\varphi(x) \notin \mathcal{V}_v \cap \mathcal{V}_w$.
Thus, there is some $u \in \Lambda$ such that $\Vert \varphi(x)-u \Vert <\Vert \varphi(x)-v \Vert$.
Since $\Vert \varphi(y_i)-u \Vert \geq \Vert \varphi(y_i)-v \Vert$ holds for $i \in \lbrace 1,2 \rbrace$,
the intermediate value theorem implies that there are $x_1 \in [y_1, x)$ and $x_2 \in (x, y_2]$ with $\Vert \varphi(x_i)-u \Vert =\Vert \varphi(x_i)-v \Vert$ for $i \in \lbrace 1,2 \rbrace$.
Hence, the bisectors $\mathcal{H}_{\Vert \cdot \Vert}^{=}(v,w)$ and $\mathcal{H}_{\Vert \cdot \Vert}^{=}(u,v)$ intersect in at least two points, which contradicts Proposition~\ref{prop:trisector_2dimensions}.
\end{proof}
Two tiles of a normal tiling are called \emph{adjacent} if their intersection contains infinitely many points.
Proposition~\ref{prop:2dimensions_strictly_convex_exact} is an immediate corollary of the following theorem and Theorem~\ref{thm:bijection_facets_Voronoi-relevant}.
\begin{theorem}[\cite{gruenbaum}, Theorem 3.2.6]
\label{thm:gruenbaum}
If $\mathcal{T}$ is a normal tiling in which each tile has the same number $k$ of adjacent tiles, then $k \in \lbrace 3,4,5,6 \rbrace.$
\end{theorem}
\begin{proofof6VR}
It follows directly from Definition~\ref{defn:facet} that the facets of the Voronoi cell $\mathcal{V}(\Lambda, \Vert \cdot \Vert)$ are exactly the intersections of this cell with its adjacent cells.
Applying Theorem~\ref{thm:gruenbaum} to the tiling in Lemma~\ref{lem:tiling} shows that $\mathcal{V}(\Lambda, \Vert \cdot \Vert)$ has between three and six facets.
The number of facets must be even, because $\mathcal{F} \subseteq \mathcal{V}(\Lambda, \Vert \cdot \Vert)$ is a facet if and only if $-\mathcal{F}$ is a facet.
Now Proposition~\ref{prop:2dimensions_strictly_convex_exact} follows from Theorem~\ref{thm:bijection_facets_Voronoi-relevant}.
\end{proofof6VR}
Already for the Euclidean norm there are two-dimensional lattices with four and with six Voronoi-relevant vectors: the lattice in Figure~\ref{fig:Voronoi_cell_1norm} has six facets with respect to the Euclidean norm, whereas the lattice spanned by $\left( \begin{smallmatrix} 1 \\ 0 \end{smallmatrix} \right)$ and $\left( \begin{smallmatrix} 0 \\ 1 \end{smallmatrix} \right)$ has only four facets.

We saw in Subsection~\ref{ssec:genNorms} that the Voronoi-revelant vectors are generally not enough to define the Voronoi cell for non-strictly convex norms, but that the weak Voronoi-relevant vectors are. Hence, we would like to have a constant upper bound for the number of the weak Voronoi-relevant vectors in two dimensions, analogously to Proposition~\ref{prop:2dimensions_strictly_convex_exact}. Unfortunately, this is not true.

\begin{proofofgenVRcomb}
  Let $\mathcal{L}_m := \mathcal{L}(b_{m,1}, b_{m,2})$ with $b_{m,1} = (1,1)^T$ and $b_{m,2} = (0,m)^T$. Furthermore, let $x := \frac{1}{2} b_{m,2}$. By considering all $v \in \mathcal{L}_m$ with $\Vert x-v \Vert_1 \leq \frac{m}{2}$, one finds two outcomes: First, there is no $v \in \mathcal{L}_m$ satisfying the strict inequality $\Vert x-v \Vert_1 < \frac{m}{2}$. Secondly, the equality $\Vert x-v \Vert_1 = \frac{m}{2}$ holds exactly for $v = z_1 b_{m,1} + z_2 b_{m,2}$ with $z_2 = 0$ and $z_1 \in [0, \frac{m}{2}] \cap \mathbb{Z}$ or $z_2 = 1$ and $z_1 \in [- \frac{m}{2},0] \cap \mathbb{Z}$. Therefore, $z_1b_{m,1}+z_2b_{m,2} \in \mathcal{L}_m$ is a weak Voronoi-relevant vector if $z_2 = 0, z_1 \in \left( 0, \frac{m}{2} \right]$ or $z_2 = 0, z_1 \in \left[-\frac{m}{2}, 0 \right)$ or $z_2 = 1, z_1 \in \left[-\frac{m}{2},0 \right]$ or $z_2 = -1, z_1 \in \left[ 0, \frac{m}{2} \right]$.
\end{proofofgenVRcomb}

\section*{Acknowledgments}
We thank Felix Dorrek for valuable discussions and the anonymous referees for very useful hints.

\bibliographystyle{siamplain}
\bibliography{references}
\end{document}